\documentclass{mcom-l}

\usepackage{subcaption}
\usepackage{enumerate}
\usepackage{hyperref}
\usepackage{xcolor}

\definecolor{azure}{RGB}{0, 127, 255}
\definecolor{signalviolet}{RGB}{161, 11, 112}
\definecolor{orange}{RGB}{243, 119, 53}

\usepackage{mathtools, amsmath, amssymb, amsthm, mathrsfs}
\usepackage{bm} \usepackage{dsfont} 

\newcommand{\transp}{^{\mathsf T}}
\newcommand{\herm}{^{\ast}}
\newcommand{\inv}{^{-1}}
\newcommand{\Id}{\mathrm{Id}}
\renewcommand{\qedsymbol}{\hspace*{0ex} \hfill \rule{1.5ex}{1.5ex} \newline}

\newcommand{\Xmz}{\bm X_{\mathrm{\tiny{MZ}}}}
\newcommand{\Wmz}{\bm W_{\mathrm{\tiny{MZ}}}}
\newcommand{\Lmz}{\bm L_{\mathrm{\tiny{MZ}}}}
\newcommand{\Imz}{I_{\mathrm{\tiny{MZ}}}}

\numberwithin{equation}{section}
\numberwithin{table}{section}
\numberwithin{figure}{section}

\DeclareMathOperator*{\spn}{span}
\DeclareMathOperator*{\diag}{diag}

\DeclareMathOperator*{\argmin}{arg\,min}

\newtheorem{theorem}{Theorem}[section]
\newtheorem{lemma}[theorem]{Lemma}
\newtheorem{remark}[theorem]{Remark}

\newtheorem{corollary}[theorem]{Corollary}
\newtheorem{proposition}[theorem]{Proposition}
\newtheorem{assumption}[theorem]{Assumption}

\begin{document} 

\title[function reconstruction from subsampled quadrature points]{On the reconstruction of functions from values at subsampled quadrature points}

\author[F.~Bartel]{Felix~Bartel}
\address{Faculty of Mathematics, Chemnitz University of Technology, D-09107 Chemnitz}
\email{felix.bartel@mathematik.tu-chemnitz.de}

\author[L.~K\"ammerer]{Lutz~K\"ammerer}
\address{Faculty of Mathematics, Chemnitz University of Technology, D-09107 Chemnitz}
\email{lutz.kaemmerer@mathematik.tu-chemnitz.de}

\author[D.~Potts]{Daniel~Potts}
\address{Faculty of Mathematics, Chemnitz University of Technology, D-09107 Chemnitz}
\email{daniel.potts@mathematik.tu-chemnitz.de}

\author[T.~Ullrich]{Tino~Ullrich}
\address{Faculty of Mathematics, Chemnitz University of Technology, D-09107 Chemnitz}
\email{tino.ullrich@mathematik.tu-chemnitz.de}



\date{}

\dedicatory{}

\begin{abstract}
This paper is concerned with function reconstruction from samples. The sampling points used in several approaches are (1) structured points connected with fast algorithms or
(2) unstructured points coming from, e.g., an initial random draw to achieve an improved information complexity.
We connect both approaches and propose a subsampling of structured points in an offline step.
In particular, we start with structured quadrature points (QMC), which provide stable $L_2$ reconstruction properties. The subsampling procedure consists of a computationally inexpensive random step followed by a deterministic procedure to further reduce the number of points while keeping its information.
In these points functions (belonging to a RKHS of bounded functions) will be sampled and reconstructed from whilst achieving state of the art error decay.

Our method is dimension-independent and is applicable as soon as we know some initial quadrature points.
We apply our general findings on the $d$-dimensional torus to subsample rank-1 lattices, where it is known that full rank-1 lattices lose half the optimal order of convergence (expressed in terms of the size of the lattice). In contrast to that, our subsampled version regains the optimal rate since many of the lattice points are not needed. Moreover, we utilize fast and memory efficient Fourier algorithms in order to compute the approximation. Numerical experiments in several dimensions support our findings.
\end{abstract}

\maketitle

\section{Introduction}
The goal of this paper is the reconstruction of multivariate functions $f\colon D\to\mathds C$, $D\subset \mathds R^d$, from sampled function values. There seems to be an increasing interest in this topic with several new and surprising approaches ranging from highly non-constructive to implementable, cf.\ \cite{KaPoVo17, ByKaUlVo16, Gr20, KUV19, KrUl19, KrUl20, Temlyakov20, TeUl22a, MoUl20, KKKS21, TeUl22, LiTe20, NaSchUl20, DKU22, BSU22, JaUlVo22}. Since it is known  \cite{ByKaUlVo16} that the highly optimized and fast algorithms \cite{CoNu07, KaPoVo17, KuoMiNoNu19, Kaemmerer20, KKKS21} are often not optimal with respect to information complexity, there is a trend to make the recent ``complexity optimal'' approaches \cite{Temlyakov20, NaSchUl20, DKU22, LiTe20} accessible for practitioners, cf.\ \cite{HNP20, BSU22}.
We continue working in this direction and provide new optimal theoretical results by exploiting a constructive subsampling strategy \cite{BSU22} on the one hand. On the other hand we focus on implementability and fast/memory-efficient algorithms. The latter is tightly connected to the use of structured sampling point sets since they often allow for fast matrix vector multiplication (FFT).
In \cite{JoKuSl13} there are several such constructions available mostly from numerical integration, e.g.\ lattice constructions, which allow for fast computation.
In addition, if a point set $\Xmz = \{\bm x^1, \dots, \bm x^M\}\subset D$ and weights $\Wmz = \diag(\omega_1, \dots, \omega_M)$ exactly integrate functions from a space of the form $V\cdot \bar{V}$, i.e., for $0<A<\infty$
\begin{align}\label{f002}
    \sum_{i=1}^M \omega_i g(\bm x^i) \overline{h(\bm x^i)}
    = A \int_D g(\bm x) \overline{h(\bm x)} \,\mathrm d\nu(\bm x)
    \quad\text{for all}\quad
    g, h\in V\,,
\end{align}
then, equivalently, they establish an $L_2$-Marcinkiewicz-Zygmund inequality with constants $A=B$ for all $f \in V$, cf.\ Theorem~\ref{lemma:mzquadrature}.
The concept of these inequalities goes back to J.~Marcinkiewicz and A.~Zygmund \cite{MZ37} and establishes a connection between the continuous $L_2$-norm and point evaluations of functions.
A set of points $\Xmz$ and weights $\Wmz$ fulfills an $L_2$-Marcinkiewicz-Zygmund (MZ) inequality with constants $0<A\le B<\infty$ for a finite-dimensional function space $V$, if
\begin{align}\label{f000}
    A\|f\|_{L_2}^2
    \le \sum_{i=1}^{M} \omega_i |f(\bm x^i)|^2
    \le B\|f\|_{L_2}^2
    \quad\text{for all}\quad
    f\in V\,.
\end{align}
Clearly, such points together with appropriate subspaces $V$ are good for sampling recovery in $L_2:=L_2(D,\nu)$, cf.\ Theorem~\ref{thm:quaderror}. For a systematic study of MZ inequalities (also for $p\neq 2$) we refer to the recent series of papers by V.N. Temlyakov and coauthors, see for instance \cite{Te18, KaKoLiTe22}. 

Note that $M$ in \eqref{f000} might be much larger than $\dim(V)$ which requires an additional ``subsampling step'' in order to create an optimal least squares approximation method out of an exact quadrature rule. By bridging to the frame terminology we are able to apply recent techniques from \cite{BSU22} in order to reduce the number of points whilst keeping their approximation power. We apply two different techniques, which we use in combination:
\begin{itemize}
\item
    In a first step we use random subsampling in order to obtain an index set $J\subset\{1,\dots,M\}$ such that $\bm X = \{\bm x^i\}_{i\in J}$ and new weights $\bm W$ fulfill an $L_2$-MZ inequality with logarithmic oversampling, i.e., for a fixed constant $C$ we have $|\bm X| \le C \dim(V)\log(\dim(V))$.
\item
    In a second step we use the deterministic \texttt{BSS}-Algorithm from \cite{BaSpSr09, BSU22} to further reduce the number of points to merely linear oversampling, i.e., for $b$ close to one we compute $J' \subset J$ such that $\bm X' = \{\bm x^i\}_{i\in J'}$ and weights $\bm W'$ with $|\bm X'| \le \lceil b \dim(V)\rceil$ fulfill an $L_2$-MZ inequality as well.
\end{itemize}

\begin{figure} \centering
\begin{minipage}{0.245\linewidth}
    \centering
    \begingroup
  \makeatletter
  \providecommand\color[2][]{\GenericError{(gnuplot) \space\space\space\@spaces}{Package color not loaded in conjunction with
      terminal option `colourtext'}{See the gnuplot documentation for explanation.}{Either use 'blacktext' in gnuplot or load the package
      color.sty in LaTeX.}\renewcommand\color[2][]{}}\providecommand\includegraphics[2][]{\GenericError{(gnuplot) \space\space\space\@spaces}{Package graphicx or graphics not loaded}{See the gnuplot documentation for explanation.}{The gnuplot epslatex terminal needs graphicx.sty or graphics.sty.}\renewcommand\includegraphics[2][]{}}\providecommand\rotatebox[2]{#2}\@ifundefined{ifGPcolor}{\newif\ifGPcolor
    \GPcolortrue
  }{}\@ifundefined{ifGPblacktext}{\newif\ifGPblacktext
    \GPblacktexttrue
  }{}\let\gplgaddtomacro\g@addto@macro
\gdef\gplbacktext{}\gdef\gplfronttext{}\makeatother
  \ifGPblacktext
\def\colorrgb#1{}\def\colorgray#1{}\else
\ifGPcolor
      \def\colorrgb#1{\color[rgb]{#1}}\def\colorgray#1{\color[gray]{#1}}\expandafter\def\csname LTw\endcsname{\color{white}}\expandafter\def\csname LTb\endcsname{\color{black}}\expandafter\def\csname LTa\endcsname{\color{black}}\expandafter\def\csname LT0\endcsname{\color[rgb]{1,0,0}}\expandafter\def\csname LT1\endcsname{\color[rgb]{0,1,0}}\expandafter\def\csname LT2\endcsname{\color[rgb]{0,0,1}}\expandafter\def\csname LT3\endcsname{\color[rgb]{1,0,1}}\expandafter\def\csname LT4\endcsname{\color[rgb]{0,1,1}}\expandafter\def\csname LT5\endcsname{\color[rgb]{1,1,0}}\expandafter\def\csname LT6\endcsname{\color[rgb]{0,0,0}}\expandafter\def\csname LT7\endcsname{\color[rgb]{1,0.3,0}}\expandafter\def\csname LT8\endcsname{\color[rgb]{0.5,0.5,0.5}}\else
\def\colorrgb#1{\color{black}}\def\colorgray#1{\color[gray]{#1}}\expandafter\def\csname LTw\endcsname{\color{white}}\expandafter\def\csname LTb\endcsname{\color{black}}\expandafter\def\csname LTa\endcsname{\color{black}}\expandafter\def\csname LT0\endcsname{\color{black}}\expandafter\def\csname LT1\endcsname{\color{black}}\expandafter\def\csname LT2\endcsname{\color{black}}\expandafter\def\csname LT3\endcsname{\color{black}}\expandafter\def\csname LT4\endcsname{\color{black}}\expandafter\def\csname LT5\endcsname{\color{black}}\expandafter\def\csname LT6\endcsname{\color{black}}\expandafter\def\csname LT7\endcsname{\color{black}}\expandafter\def\csname LT8\endcsname{\color{black}}\fi
  \fi
    \setlength{\unitlength}{0.0500bp}\ifx\gptboxheight\undefined \newlength{\gptboxheight}\newlength{\gptboxwidth}\newsavebox{\gptboxtext}\fi \setlength{\fboxrule}{0.5pt}\setlength{\fboxsep}{1pt}\definecolor{tbcol}{rgb}{1,1,1}\begin{picture}(1700.00,1700.00)\gplgaddtomacro\gplbacktext{}\gplgaddtomacro\gplfronttext{}\gplbacktext
    \put(0,0){\includegraphics[width={85.00bp},height={85.00bp}]{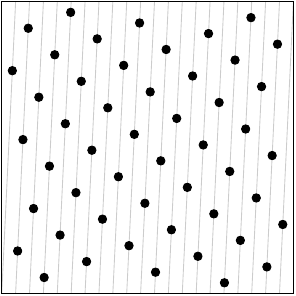}}\gplfronttext
  \end{picture}\endgroup
 \\
    $\Xmz \!=\! \{\bm x^1\!, \dots, \bm x^M\!\}$\\
    $\Wmz \in [0,\infty)^{M\times M}$
\end{minipage}
\begin{minipage}{0.11\linewidth}
    \centering
    \raisebox{1.5cm}{$\overset{\text{random}}{\underset{\text{subsampling}}{\rightarrow}}$}
\end{minipage}
\begin{minipage}{0.245\linewidth}
    \centering
    \begingroup
  \makeatletter
  \providecommand\color[2][]{\GenericError{(gnuplot) \space\space\space\@spaces}{Package color not loaded in conjunction with
      terminal option `colourtext'}{See the gnuplot documentation for explanation.}{Either use 'blacktext' in gnuplot or load the package
      color.sty in LaTeX.}\renewcommand\color[2][]{}}\providecommand\includegraphics[2][]{\GenericError{(gnuplot) \space\space\space\@spaces}{Package graphicx or graphics not loaded}{See the gnuplot documentation for explanation.}{The gnuplot epslatex terminal needs graphicx.sty or graphics.sty.}\renewcommand\includegraphics[2][]{}}\providecommand\rotatebox[2]{#2}\@ifundefined{ifGPcolor}{\newif\ifGPcolor
    \GPcolortrue
  }{}\@ifundefined{ifGPblacktext}{\newif\ifGPblacktext
    \GPblacktexttrue
  }{}\let\gplgaddtomacro\g@addto@macro
\gdef\gplbacktext{}\gdef\gplfronttext{}\makeatother
  \ifGPblacktext
\def\colorrgb#1{}\def\colorgray#1{}\else
\ifGPcolor
      \def\colorrgb#1{\color[rgb]{#1}}\def\colorgray#1{\color[gray]{#1}}\expandafter\def\csname LTw\endcsname{\color{white}}\expandafter\def\csname LTb\endcsname{\color{black}}\expandafter\def\csname LTa\endcsname{\color{black}}\expandafter\def\csname LT0\endcsname{\color[rgb]{1,0,0}}\expandafter\def\csname LT1\endcsname{\color[rgb]{0,1,0}}\expandafter\def\csname LT2\endcsname{\color[rgb]{0,0,1}}\expandafter\def\csname LT3\endcsname{\color[rgb]{1,0,1}}\expandafter\def\csname LT4\endcsname{\color[rgb]{0,1,1}}\expandafter\def\csname LT5\endcsname{\color[rgb]{1,1,0}}\expandafter\def\csname LT6\endcsname{\color[rgb]{0,0,0}}\expandafter\def\csname LT7\endcsname{\color[rgb]{1,0.3,0}}\expandafter\def\csname LT8\endcsname{\color[rgb]{0.5,0.5,0.5}}\else
\def\colorrgb#1{\color{black}}\def\colorgray#1{\color[gray]{#1}}\expandafter\def\csname LTw\endcsname{\color{white}}\expandafter\def\csname LTb\endcsname{\color{black}}\expandafter\def\csname LTa\endcsname{\color{black}}\expandafter\def\csname LT0\endcsname{\color{black}}\expandafter\def\csname LT1\endcsname{\color{black}}\expandafter\def\csname LT2\endcsname{\color{black}}\expandafter\def\csname LT3\endcsname{\color{black}}\expandafter\def\csname LT4\endcsname{\color{black}}\expandafter\def\csname LT5\endcsname{\color{black}}\expandafter\def\csname LT6\endcsname{\color{black}}\expandafter\def\csname LT7\endcsname{\color{black}}\expandafter\def\csname LT8\endcsname{\color{black}}\fi
  \fi
    \setlength{\unitlength}{0.0500bp}\ifx\gptboxheight\undefined \newlength{\gptboxheight}\newlength{\gptboxwidth}\newsavebox{\gptboxtext}\fi \setlength{\fboxrule}{0.5pt}\setlength{\fboxsep}{1pt}\definecolor{tbcol}{rgb}{1,1,1}\begin{picture}(1700.00,1700.00)\gplgaddtomacro\gplbacktext{}\gplgaddtomacro\gplfronttext{}\gplbacktext
    \put(0,0){\includegraphics[width={85.00bp},height={85.00bp}]{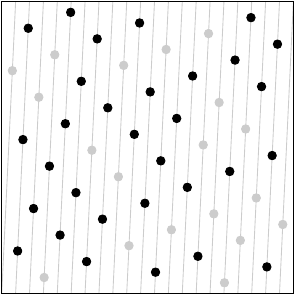}}\gplfronttext
  \end{picture}\endgroup
     $\bm X = \{\bm x^i\}_{i\in J}$\\
    $\bm W\in [0,\infty)^{|\bm X|\times |\bm X|}$
\end{minipage}
\begin{minipage}{0.11\linewidth}
    \centering
    \raisebox{1.5cm}{$\overset{\text{BSS}}{\underset{\text{subsampling}}{\rightarrow}}$}
\end{minipage}
\begin{minipage}{0.245\linewidth}
    \centering
    \begingroup
  \makeatletter
  \providecommand\color[2][]{\GenericError{(gnuplot) \space\space\space\@spaces}{Package color not loaded in conjunction with
      terminal option `colourtext'}{See the gnuplot documentation for explanation.}{Either use 'blacktext' in gnuplot or load the package
      color.sty in LaTeX.}\renewcommand\color[2][]{}}\providecommand\includegraphics[2][]{\GenericError{(gnuplot) \space\space\space\@spaces}{Package graphicx or graphics not loaded}{See the gnuplot documentation for explanation.}{The gnuplot epslatex terminal needs graphicx.sty or graphics.sty.}\renewcommand\includegraphics[2][]{}}\providecommand\rotatebox[2]{#2}\@ifundefined{ifGPcolor}{\newif\ifGPcolor
    \GPcolortrue
  }{}\@ifundefined{ifGPblacktext}{\newif\ifGPblacktext
    \GPblacktexttrue
  }{}\let\gplgaddtomacro\g@addto@macro
\gdef\gplbacktext{}\gdef\gplfronttext{}\makeatother
  \ifGPblacktext
\def\colorrgb#1{}\def\colorgray#1{}\else
\ifGPcolor
      \def\colorrgb#1{\color[rgb]{#1}}\def\colorgray#1{\color[gray]{#1}}\expandafter\def\csname LTw\endcsname{\color{white}}\expandafter\def\csname LTb\endcsname{\color{black}}\expandafter\def\csname LTa\endcsname{\color{black}}\expandafter\def\csname LT0\endcsname{\color[rgb]{1,0,0}}\expandafter\def\csname LT1\endcsname{\color[rgb]{0,1,0}}\expandafter\def\csname LT2\endcsname{\color[rgb]{0,0,1}}\expandafter\def\csname LT3\endcsname{\color[rgb]{1,0,1}}\expandafter\def\csname LT4\endcsname{\color[rgb]{0,1,1}}\expandafter\def\csname LT5\endcsname{\color[rgb]{1,1,0}}\expandafter\def\csname LT6\endcsname{\color[rgb]{0,0,0}}\expandafter\def\csname LT7\endcsname{\color[rgb]{1,0.3,0}}\expandafter\def\csname LT8\endcsname{\color[rgb]{0.5,0.5,0.5}}\else
\def\colorrgb#1{\color{black}}\def\colorgray#1{\color[gray]{#1}}\expandafter\def\csname LTw\endcsname{\color{white}}\expandafter\def\csname LTb\endcsname{\color{black}}\expandafter\def\csname LTa\endcsname{\color{black}}\expandafter\def\csname LT0\endcsname{\color{black}}\expandafter\def\csname LT1\endcsname{\color{black}}\expandafter\def\csname LT2\endcsname{\color{black}}\expandafter\def\csname LT3\endcsname{\color{black}}\expandafter\def\csname LT4\endcsname{\color{black}}\expandafter\def\csname LT5\endcsname{\color{black}}\expandafter\def\csname LT6\endcsname{\color{black}}\expandafter\def\csname LT7\endcsname{\color{black}}\expandafter\def\csname LT8\endcsname{\color{black}}\fi
  \fi
    \setlength{\unitlength}{0.0500bp}\ifx\gptboxheight\undefined \newlength{\gptboxheight}\newlength{\gptboxwidth}\newsavebox{\gptboxtext}\fi \setlength{\fboxrule}{0.5pt}\setlength{\fboxsep}{1pt}\definecolor{tbcol}{rgb}{1,1,1}\begin{picture}(1700.00,1700.00)\gplgaddtomacro\gplbacktext{}\gplgaddtomacro\gplfronttext{}\gplbacktext
    \put(0,0){\includegraphics[width={85.00bp},height={85.00bp}]{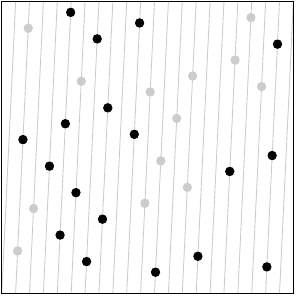}}\gplfronttext
  \end{picture}\endgroup
     $\bm X' = \{\bm x^i\}_{i\in J'}$\\
    $\bm W'\!\in\! [0,\!\infty)^{|\bm X'|\times |\bm X'|}$
\end{minipage}
\caption{Two-step subsampling procedure from a rank-1 lattice}\label{fig:scheme}
\end{figure} 

This scheme is depicted in Figure~\ref{fig:scheme} for initial points from a rank-1 lattice, see Section~\ref{sec:t} for details.
The corresponding results are found in Theorems~\ref{thm:randomsubsample_finite}, \ref{thm:bsssubsample_finite_weighted}, and \ref{thm:bsssubsample_finite}.
There is freedom in the way of subsampling which allows to control the approximation error for functions from infinite-dimensional function spaces as well.
We assume to be given a class of functions $H(K)$ (reproducing kernel Hilbert space with kernel $K$ on $D$) and an initial set of $L_2$-MZ points $\Xmz$ for a subspace $V = \spn\{\eta_k\}_{k\in\Imz} \subset H(K)$ spanned by $L_2$-orthonormal eigenfunctions $\eta_k(\cdot)$ of the operator $W = \Id^\ast\circ\Id$ with $\Id\colon H(K)\hookrightarrow L_2$ the embedding.
We then consider the least squares reconstruction $S_{I}^{\bm X}$, defined in \eqref{eq:lsqr}, on the possibly smaller subspace $V = \spn\{\eta_k\}_{k\in I}$, $I\subset\Imz$, where the choice of $I$ is a degree of freedom, and are interested in the $L_2$-worst-case error
\begin{equation}\label{f001}
    \sup_{\|f\|_{H(K)}\le 1} \|f-S_{I}^{\bm X}f\|_{L_2}^2\,.
\end{equation}
The least squares approximation $S_{\bm X}^{I} f$ utilizes only points evaluations and is a discrete version of the $L_2$-projection onto $V$
\begin{align}\label{eq:projection}
    P_{I} f
    = \argmin_{g\in V} \|f-g\|_{L_2}^2
    = \sum_{k\in I} \langle f, \eta_k \rangle_{L_2} \eta_k \,,
\end{align}
where the last equality holds since the $\eta_k$ are $L_2$-orthonormal.
The projection is optimal and will serve as a benchmark.

In Theorem \ref{thm:general} we show bounds for the worst-case error of the least squares approximation using a random subsample $\bm X = \{\bm x^i\}_{i\in J} \subset \Xmz$.
For this subsampling we use a special discrete probability distribution given in \eqref{eq:density} below. The bounds are stated in terms of the eigenvalues $\lambda_k$  and the approximation power of the initial method $S^{\Xmz}_{\Imz}$ expressed by a quantity like \eqref{f001}. Moreover, using an additional deterministic subsampling which yields $\bm X'\subset \bm X$ with merely linear oversampling, i.e., $|X'| \le b|I|$ for $b$ close to one, a special case of the main result in Theorem~\ref{thm:generalbss} for finite measures $\nu$ and kernels $K$ which are constant on the diagonal reads as follows.

\begin{theorem} Let $H(K)$ be a separable reproducing kernel Hilbert space with a finite trace \eqref{eq:finitetrace} and bounded \eqref{eq:boundedkernel} kernel.
    Further, let $\lambda_k$ and $\eta_k$ denote the eigenvalues and $L_2$-orthonormal eigenfunctions of the operator $W = \Id^\ast\circ\Id$ with $\Id\colon H(K)\hookrightarrow L_2$ the embedding.
    For $V = \spn\{\eta_k\}_{k\in\Imz} \subset H(K)$ a finite-dimensional function space, we assume the points $\Xmz\subset D$ and weights $\Wmz\in [0,\infty)^{M\times M}$ fulfill an $L_2$-MZ inequality for $V$ with constants $A$ and $B$.
    Finally, let $I \subset  \Imz \subset \mathbb N$, $|I|\ge 3$, be an index set and $\log(73B/A)+2\le r$ a constant.
    
    For a target oversampling factor $b > 1+\frac{1}{|I|}$ we construct a subset of points $\bm X' = \{\bm x^i\}_{i\in J'} \subset \Xmz$ with $|\bm X'| \le \lceil b|I|\rceil$ such that we have the following bound with probability $1-4/|I|$
    \begin{align*}
        \sup_{\|f\|_{H(K)}\le 1} \|f-S_{I}^{\bm X'}f\|_{L_2}^2
        &\le
        79\,616\,\frac{B}{A} \frac{(b+1)^2}{(b-1)^3} \log|I|
        \Bigg(
        \frac{B}{A} \sup_{ k\notin I}\lambda_k
        +
        \frac{r}{|I|}\sum_{k\in\Imz\setminus I} \lambda_k
        \\
        &\quad + \frac{1}{A} \sup_{\|f\|_{H(K)}\le 1} \|f-P_{\Imz}f\|_{\ell_\infty(D)}^2
        \Bigg)
    \end{align*}
    with weights $\bm W' = \diag(\omega_i/\varrho_i)_{i\in J'}$ and the least squares approximation $S_{I}^{\bm X'}f = S_{I}^{\bm X'}(\bm W')f$ defined in \eqref{eq:lsqr}.
\end{theorem} 

To interpret the above result, we assume the index set $I$ contains the indices corresponding to the largest eigenvalues and $b$ fixed. Up to a logarithmic and constant factor we obtain two categories of summands
\begin{itemize}
    \item
        The summand $\sup_{k\notin I} \lambda_k$ is the truncation error, i.e., the best possible error from the function space $\spn\{\eta_k\}_{k\in I}$.
        The summand $\frac{1}{|I|}\sum_{k\notin I} \lambda_k$ is of the same order whenever the eigenvalues have polynomial decay smaller in order than $-1$, i.e., $\lambda_k \lesssim k^{-1-\varepsilon}$.
        These only depend on the index set $I$ for the least squares approximation.
    \item
        The remaining summand is the error from the initial $L_2$-MZ inequality of the function space $V = \spn\{\eta_k\}_{k\in\Imz}$.
        Note, that the $\ell_\infty$-estimate occurs in the theory but in practice we observe a better error as we will discuss in \ref{subsec:t_numerics}.
\end{itemize}

Despite the fact that we restrict ourselves to structured nodes we achieve the same rates as in  \cite{NaSchUl20, BSU22} which are worse by the  square root of a logarithm compared to \cite{DKU22}. When going to specific scenarios (see Corollary \ref{cor:r1} below), these structured nodes  come from rank one lattices which can be well-processed in a computer (they are at least rational). The lower bound in \cite{HiKrNoVy22} indicates that the bound $1/|I|\sum_{k\notin I}\lambda_k$ in \cite{DKU22} is most likely sharp and we are close to optimal. 
In fact, the term only depending on $\Imz$ can be made arbitrarily small by choosing a good initial $L_2$-MZ inequality for large enough $\Imz$ without deteriorating the remaining bound or the size of the subsampled points $\bm X'$. Note that often $\Imz = I$ is a valid choice. 
We stress again that, by choosing a ``good enough'' initial $L_2$-MZ inequality, one achieves the best possible error rate up to a logarithm.
Recent progress has shown that the logarithmic factor might be avoided but this utilizes the Kadison-Singer theorem and is not constructive in contrast to our approach, cf.\ \cite{DKU22}.

The novelty of our result is that we start with discrete $L_2$-MZ points and positive weights which often have a structure that can be exploited towards fast matrix-vector multiplications.
These are available for, e.g.\ the $d$-dimensional torus $\mathds T^d$ with the trigonometric polynomials with Fourier-like algorithms, the cube with Chebyshev polynomials \cite{P03}, the two-dimensional sphere $\mathds S^2$ with the spherical harmonics \cite{KP03}, or the rotation group $SO(3)$ with the Wigner-D functions \cite{PPV09}.
We make a case study on the $d$-dimensional torus $\mathds T^d$ with rank-1 lattices, where we apply our theory and conduct numerical experiments.
It is known that full rank-1 lattices achieve only half the optimal main rate for Sobolev spaces with dominating mixed smoothness, cf.\ \cite{ByKaUlVo16} and Theorem~\ref{thm:r1}.
With the subsampling approach in this paper, we regain the optimal order of convergence via subsampling:

\begin{corollary}\label{cor:r1} Let $H^{s}_{\textnormal{mix}}(\mathbb{T}^d)$, $s>1/2$, be the Sobolev space with dominating mixed smoothness \eqref{eq:hsmix} on the $d$-torus, $I \subset  \Imz \subset \mathbb Z^d$, $|I| \ge 3$, be the hyperbolic cross frequency index sets defined in \eqref{eq:hc}, and $\Xmz$ a reconstructing rank-1 lattice for $\Imz$ with $M$ points, cf.\ \eqref{eq:reconstructingproperty}.
    For $b > 1+\frac{1}{|I|}$ we construct a subset of points $\bm X' = \{\bm x^i\}_{i\in J'} \subset \Xmz$ with $|\bm X'| \le \lceil b|I|\rceil$ such that we have the following bound for the plain least squares approximation
    \begin{align}\label{Cor1.2}
        \sup_{\|f\|_{H^{s}_{\textnormal{mix}}}\le 1} \|f-S_{I}^{\bm X'}f\|_{L_2}^2
        &\le
        C_{d,s}
        \Big(\frac{b}{b-1}\Big)^3 \log|I|
        \Big(
        |I|^{-2s}(\log |I|)^{(d-1)2s} \\
        \nonumber&\quad + |\Imz|^{-2s+1}(\log|\Imz|)^{2(d-1)s}
        \Big) \,,
    \end{align}
    with probability $1-4/|I|$ and $C_{d,s}$ a constant depending on $d$ and $s$.
\end{corollary} 

Note, that the size $M$ of the initial reconstructing rank-1 lattice for $\Imz$ has to be roughly quadratic in $|\Imz|$, cf.\ Lemma~\ref{thm:r1} (iii). We work with merely linear oversampling ($n = |\bm X'| \le \lceil b|I|\rceil$) achieving the optimal error decay up to a logarithmic factor.
There exists an initial lattice of size $M$ such that the bound in \eqref{Cor1.2} above, stated in terms of the number of sample points $n$ and initial lattice size $M$, is of order less than
\begin{equation*}
        \log n\Big(n^{-2s}(\log n)^{(d-1)2s}+M^{-s+1/2}(\log M)^{(d-1)(s+1)-s}\Big)
\end{equation*}
using Theorem~\ref{thm:r1}.
Whereas full rank-$1$ lattices achieve only half the main order of convergence, the subsampled rank-$1$ lattice achieves the optimal main order of convergence.
Moreover, we are still able to utilize the well-tried memory efficient fast Fourier algorithms with the low number of samples.

For the structure of this paper we investigate $L_2$-MZ points in Section~\ref{sec:lsqr} and show their relation to exact quadrature and establish a connection to finite frames.
In Section~\ref{sec:submz} we use this frame characterization to apply different subsampling techniques in order to reduce the number of points and obtain a new $L_2$-MZ inequality.
The infinite-dimensional case is considered in Section~\ref{sec:general}, where we modify the subsampling such that we control and improve on the approximation error.
We finish with an application in Section~\ref{sec:t} on the $d$-dimensional torus $\mathds T^d$ with rank-1 lattices and conduct numerical experiments. 
\section{Exact quadrature, \texorpdfstring{$L_2$}{L2}-MZ inequalities, and least squares}\label{sec:lsqr}
The problem in this section is the recovery of functions $f\colon D\to\mathds C$ defined on some domain $D\subset\mathds R^d$.
For that we have given samples $\bm f = (f(\bm x^1), \dots, f(\bm x^M))\transp \in \mathds C^M$ in points $\Xmz = \{\bm x^1, \dots, \bm x^M\} \subset D$.
For now, we assume $f$ to be an element of a finite-dimensional function space $V\subset L_2$ with dimension smaller than $M$ (oversampling).
Our method of choice to recover the function from its samples is least squares approximation and, as we will see, $L_2$-MZ inequalities come into play in a natural way.

At first, for a $\sigma$-finite measure $\nu : D\to\mathds R$, we define the space of square-integrable functions $L_2 = L_2(D,\nu)$ with
\begin{align*}
    \langle f,g\rangle_{L_2}
    \coloneqq \int_D f(\bm x)\overline{g(\bm x)} \,\mathrm d\nu(\bm x)
    \quad\text{and}\quad
    \|f\|_{L_2}
    \coloneqq \Big( \int_D |f(\bm x)|^2 \,\mathrm d\nu(\bm x)\Big)^{1/2}.
\end{align*}
Having an inner product we calculate an orthonormal basis $\{ \eta_k \}_{k\in \Imz}$ of $V$ and, with that, have an expansion for functions $f\in V$:
\begin{align}\label{eq:fexpand}
    f(\bm x) = \sum_{k\in \Imz} a_k \eta_k(\bm x)
    \quad\text{with}\quad
    \bm a = (a_k)_{k\in \Imz} = (\langle f, \eta_k\rangle_{L_2})_{k\in \Imz}.
\end{align}
Now we seek a way to recover the coefficients $\bm a$ from the samples $\bm f$.
The following system of linear equations establishes this connection 
\begin{align}\label{eq:lsqrmatrix}
    \bm L \bm{a} = \bm{f}
    \quad\text{with}\quad
    \bm L \coloneqq
    \bm L_{\Xmz}^{\Imz}
    = \begin{pmatrix}
        \eta_1(\bm x^1) & \cdots & \eta_{|\Imz|}(\bm x^1)\\
        \vdots & \ddots & \vdots\\
        \eta_1(\bm x^M) & \cdots & \eta_{|\Imz|}(\bm x^M)
    \end{pmatrix}.
\end{align}
Assuming $\bm L$ has full rank, a stable way to solve this overdetermined system of equations is to minimize the residual $\|\bm L\bm a - \bm f\|_2^2$.
The solution is obtained by solving the normal equation and the approximation $S_{\Imz}^{\Xmz} f$ takes the form
\begin{align*}
    (S_{\Imz}^{\Xmz} f)(\bm x)
    \coloneqq \sum_{k\in \Imz}a_k \eta_k(\bm x)
    \quad\text{with}\quad
    \bm a = (\bm L^\ast\bm L)\inv\bm L^\ast \bm f.
\end{align*}

Further, we consider weighted versions of the least squares algorithm.
For a diagonal matrix $\Wmz = \diag(\omega_1, \dots, \omega_M)\in[0,\infty)^{M\times M}$ the weighted least squares approximation $S_{\Imz}^{\Xmz} f = S_{\Imz}^{\Xmz} (\Wmz) f$ takes the form
\begin{align}\label{eq:lsqr}
    (S_{\Imz}^{\Xmz} f)(\bm x)
    = \sum_{k\in \Imz} a_k \eta_k(\bm x)
    \quad\text{with}\quad
    \bm{a}
    = (\bm L^\ast\Wmz\bm L)\inv\bm L^\ast\Wmz \bm f
\end{align}
where $\bm{a}$ is the minimizer of $\bm a \mapsto \|\bm L\bm a - \bm f\|_{\Wmz}^2 = \|\Wmz^{1/2}(\bm L\bm a - \bm f)\|_2^2$.
So it is obtained from the unweighted least squares version by a change of measure using the discrete weights $\Wmz$.

If the singular values of $\Wmz^{1/2}\bm L$ are non-zero, we immediately obtain that the least squares algorithm $S_{\Imz}^{\Xmz}$ reconstructs all functions $f\in V$.
The stability of the reconstruction as well as the number of iterations 
of a conjugate gradient method used to solve the system of equations heavily depends on the condition number of the matrix $(\bm L\herm\Wmz\bm L)\inv\bm L\herm\Wmz^{1/2}$, cf.\ \cite[Theorem~3.1.1]{gre}.

\begin{lemma}\label{prop2}\cite[Proposition 3.1]{KUV19} Let $\Wmz^{1/2}\bm L\in \mathbb C^{M\times |\Imz|}$ be a matrix with $|\Imz|\leq M$ with non-zero singular values.
    Then
    \begin{flalign*}
        && \sigma_{\max} \Big((\bm L\herm\Wmz\bm L)\inv\bm L\herm\Wmz^{1/2}\Big) &= \frac{1}{\sigma_{\min}(\Wmz^{1/2}\bm L)} && \\
        \text{and} && \sigma_{\min} \Big((\bm L\herm\Wmz\bm L)\inv\bm L\herm\Wmz^{1/2}\Big) &= \frac{1}{\sigma_{\max}(\Wmz^{1/2}\bm L)} \,. &&
    \end{flalign*}
\end{lemma} 

This motivates having a look at lower and upper bounds for the singular values of $\Wmz^{1/2}\bm L$.
The following lemma gives different characterizations for these bounds.

\begin{lemma}\label{lemma:character} Let $\eta_k\colon D\to\mathds C$, $k\in \Imz$ be an $L_2$-orthonormal basis of a finite-dimensional function space $V$, let $\Xmz = \{\bm x^1, \dots, \bm x^M\}\subset D$ be points and $\Wmz = \diag(\omega_1, \dots, \omega_M)\in[0,\infty)^{M\times M}$ weights.
    Then the following are equivalent:
    \begin{enumerate}[(i)]
    \item
        the singular values of the matrix $\Wmz^{1/2}\bm L \in \mathds C^{M\times|\Imz|}$, where $\bm L$ is defined in \eqref{eq:lsqrmatrix}, lie in the interval $[\sqrt A, \sqrt B]$, i.e.,
        \begin{align*}
            A \|\bm a\|_2^2
            \le \|\Wmz^{1/2}\bm L\bm a\|_2^2
            \le B \|\bm a\|_2^2
            \quad\text{for all}\quad
            \bm a \in\mathds C^{|\Imz|} \,;
        \end{align*}
    \item
        the rows $\sqrt{\omega_i}(\eta_k(\bm x^i))_{k\in \Imz} \in \mathds C^{|\Imz|}$, $i=1,\dots,M$, of the matrix $\Wmz^{1/2}\bm L$ form a frame with bounds $A$ and $B$, i.e.,
        \begin{align*}
            A \|\bm a\|_2^2
            \le \sum_{i=1}^M |\langle\bm a, \sqrt{\omega_i}(\eta_k(\bm x^i))_{k\in \Imz}\rangle|^2
            \le B \|\bm a\|_2^2
            \quad\text{for all}\quad
            \bm a \in\mathds C^{|\Imz|}\,;
        \end{align*}
    \item
        the points $\Xmz$ and weights $\Wmz$ form an $L_2$-MZ inequality for the function space $V = \spn\{\eta_k\}_{k\in \Imz}$ with bounds $A$ and $B$, i.e.,
        \begin{align}\label{eq:mz}
            A\|f\|_{L_2}^2
            \le \sum_{i=1}^M \omega_i |f(\bm x^i)|^2
            \le B\|f\|_{L_2}^2
            \quad\text{for all}\quad
            f\in V\,.
        \end{align}
    \end{enumerate}
\end{lemma} 

\begin{proof} \textbf{Step 1.} To show the equivalence of (i) and (ii) we rewrite the matrix-vector product $\Wmz^{1/2}\bm L\overline{\bm a}$ using the Euclidian inner products of $\bm a$ and the rows of $\Wmz^{1/2}\bm L$ as follows
\begin{align*}
    A \|\bm a\|_2^2
    =\! A \|\overline{\bm a}\|_2^2
    \le \| \Wmz^{1/2}\bm L\overline{\bm a} \|_2^2
    =\! \sum_{i=1}^M |\langle\bm a, \sqrt{\omega_i}(\eta_k(\bm x^i))_{k\in \Imz}\rangle_2|^2
    \le\! B \|\overline{\bm a}\|_2^2=B \|\bm a\|_2^2\,.
\end{align*}
This immediately shows the equivalence of the two stated conditions.

\textbf{Step 2.} Now we show the equivalence of (ii) and (iii).
Using the series expansion \eqref{eq:fexpand} of $f$, we have
\begin{align*}
    \sum_{i=1}^M |\langle\overline{\bm a}, \sqrt{\omega_i}(\eta_k(\bm x^i))_{k\in \Imz}\rangle_2|^2
    = \sum_{i=1}^M \omega_i \Big| \sum_{k\in \Imz} a_k \eta_k(\bm x^i) \Big|^2 
    = \sum_{i=1}^M \omega_i |f(\bm x^i)|^2
\end{align*}
and further using Parseval's identity, we obtain
\begin{align*}
    \|\overline{\bm a}\|_2^2 =\|\bm a\|_2^2 
    = \Big\| \sum_{k\in \Imz} a_k \eta_k(\bm x^i) \Big\|_{L_2}^2
    = \|f\|_{L_2}^2\,.
\end{align*}
Plugging these two formulas into the frame condition, we obtain a reformulation in terms of functions, which is an $L_2$-MZ inequality for the function space $V$ with bounds $A$ and $B$.
\end{proof} 

The $L_2$-MZ characterization supports the availability of points with well-behaved least squares matrices.
With Lemma~\ref{lemma:character}, using a set of points $\Xmz$ coming from an $L_2$-MZ inequality with well-behaved constants, we automatically have a good point set to use in least squares approximation, and vice versa.
Point sets fulfilling $L_2$-MZ inequalities are widely available and well-studied, cf.~\cite{MhNaWa01, KeKuPo07, FiMh11, CDL13, CM17, Te18}.
Furthermore, the next theorem shows an equivalence to an exact integration condition when the constants in the $L_2$-MZ inequality coincide.
This widens the applicability even further, cf.~\cite{CoNu07, Tref13, KaPoVo13}.

\begin{theorem}\label{lemma:mzquadrature} The points $\Xmz=\{\bm x^1,\dots,\bm x^M\}\subset D$ and weights $\Wmz = \diag(\omega_1, \dots, \omega_M)\in [0,\infty)^{M\times M}$ obey an $L_2$-MZ inequality \eqref{eq:mz} on the function space $V = \spn\{\eta_k\}_{k\in \Imz}$ with constants $A=B$ if and only if we have exact quadrature on $V \cdot \overline{V} \coloneqq \{f = \overline{g}\cdot h : g, h\in V\}$, i.e., we have
    \begin{align*}
        \sum_{i=1}^M \omega_i g(\bm x^i) \overline{h(\bm x^i)}
        = A \int_D g(\bm x) \overline{h(\bm x)} \,\mathrm d\nu(\bm x)
        \quad\text{for all}\quad
        g, h\in V\,.
    \end{align*}
\end{theorem} 

\begin{proof} Starting with the $L_2$-MZ inequality, we have 
    \begin{equation*}
        \sum_{i=1}^{M} \omega_i |f(\bm x^i)|^2
        = A \int_D |f(\bm x)|^2 \;\mathrm d\nu(\bm x)
        \quad\text{for all}\quad
        f\in V \,.
    \end{equation*}
    By the parallelogram law the corresponding inner products also coincide, which is on direction of the assertion.
    The reverse is achieved by using $g = h$.
\end{proof} 

 \section{Subsampling of \texorpdfstring{$L_2$}{L2}-MZ inequalities}\label{sec:submz}
In this section we utilize recent random and deterministic frame subsampling techniques from \cite{BSU22} to select subsets of $L_2$-MZ point sets keeping their good approximation properties following the illustration in Figure~\ref{fig:scheme}.
The connection from $L_2$-MZ points to frames was already done in Lemma~\ref{lemma:character}.
The best possible start are tight frames, i.e., with frame bounds $A=B$, which is equivalent to having exact quadrature as we have seen in Theorem~\ref{lemma:mzquadrature}.

The next theorem covers random subsampling.

\begin{theorem}\label{thm:randomsubsample_finite} Let $0 < C \le 1$, $\{\eta_k\}_{k\in \Imz}$ be an $L_2$-orthonormal basis of the finite-dimensional function space $V$, let $\bm X = \{\bm x^1, \dots, \bm x^M\}\subset D$ be points and $\Wmz = \diag(\omega_1, \dots, \omega_M)\in[0,\infty)^{M\times M}$ weights fulfilling an $L_2$-MZ inequality for $V$ with constants $A$ and $B$.
    Further, let $t>0$, $I\subset\Imz$, and $n\in\mathds N$ be such that
    \begin{align}
        n \ge \frac{12B}{A C} |I| (\log |I|+t) \,.\label{equ:rand_subsampling_n}
    \end{align}
    We draw a set $\bm X = \{\bm x^i\}_{i\in J}$, $|J| = n$, of points i.i.d.\ from $\Xmz$ with respect to the discrete probability weights $\varrho_i$ (with duplicates), which fulfill
    \begin{align}
        \varrho_i \ge \frac{C \omega_i \sum_{k\in I}|\eta_k(\bm x^i)|^2}{\sum_{j=1}^{M} \omega_j \sum_{k\in I}|\eta_k(\bm x^j)|^2}
        \quad\text{for}\quad
        i=1,\dots,M\,.\label{eq:prob_density_christoffel_weighted_least_squares}
    \end{align}
    Then, with probability larger than $1-2\exp(-t)$, there holds the subsampled $L_2$-MZ inequality
    \begin{align}\label{eq:mz_subs}
        \frac{1}{2}A\|f\|_{L_2}^2
        \le \sum_{i \in J} \frac{\omega_i}{n\varrho_i} |f(\bm x^i)|^2
        \le \frac{3}{2}B\|f\|_{L_2}^2
        \quad\text{for all}\quad
        f\in \spn\{\eta_k\}_{k\in I}\,.
    \end{align}
\end{theorem} 

\begin{proof} We will use \cite[Theorem~A.3]{BSU22} (which is a reformulation of \cite[Theorem~1.1]{Tr12}) with $\bm A_i = \frac{\omega_i}{n\varrho_i} \bm y^i\otimes\bm y^i$, $\bm y^i = (\eta_k(\bm x^i))^{k\in I}$ and $t=1/2$ (this is not the same $t$ as in the present assertion).
    By the property of $\varrho_i$ and the given $L_2$-MZ inequality, we have
    \begin{equation*}
        \lambda_{\max}(\bm A_i)
        = \frac{\omega_i}{n\varrho_i} \|\bm y^i\|_2^2
        \le \frac{1}{n C} \sum_{j=1}^{M} \omega_j \sum_{k\in I}|\eta_k(\bm x^j)|^2
        \le \frac{B|I|}{n C}\,.
    \end{equation*}
    Further,
    \begin{equation*}
        \sum_{i=1}^{n} \mathds E(\bm A_i)
        = \sum_{i=1}^{n} \sum_{j=1}^{M} \varrho_j \frac{\omega_j}{n\varrho_j} \bm y^j\otimes\bm y^j
        = \Lmz^\ast \Wmz\Lmz \,,
    \end{equation*}
    where $\Lmz$ is the system matrix with the points $\Xmz$ and $\Wmz = \diag(\omega_1, \dots, \omega_M)$.
    With Lemma~\ref{lemma:character}, the eigenvalues of the matrix $\Lmz^\ast\Wmz\Lmz$ are bounded by $A\le\lambda_{\min}(\Lmz^\ast\Wmz\Lmz) \le \lambda_{\max}(\Lmz^\ast\Wmz\Lmz) \le B$.
    Note, $A/2 \le \lambda_{\min}(\sum_{i=1}^{n} \bm A_i)$ and $\lambda_{\max}(\sum_{i=1}^{n} \bm A_i) \le 3B/2$ are the target inequalities.
    
    By applying \cite[Theorem~A.3]{BSU22} and using the assumption on $n$, both inequalities of the assertion hold with probability
    \begin{equation*}
        1-|I| \exp\Big(
        -\frac{An}{12CB|I|}
        \Big)
        \ge 1-\exp(-t) \,.
    \end{equation*}
    Union bound gives the overall probability and, thus, the assertion.
\end{proof} 

Thus, we found a subset $\bm X$ of $\Xmz$ with logarithmic oversampling (independent of the original size $M$) which fulfills an $L_2$-MZ inequality with similar constants.
The probability density stated in \eqref{eq:prob_density_christoffel_weighted_least_squares} is also used in Christoffel-weighted least squares, cf.\ \cite{NJZ17}.

\begin{remark} We could apply the orthogonalization trick from \cite[Lemma~4.3]{BSU22} to eliminate the factor $B/A$ in the assumption on the number of points $n$ in \eqref{equ:rand_subsampling_n}.
    But this demands for setting up the matrix $\bm L$ from \eqref{eq:lsqrmatrix} whereas in the above formulation, we only need the evaluation of the Christoffel function $N(I) = \sum_{k\in I} |\eta_k(\bm x)|^2$.
\end{remark} 

Next, we subsample $\bm X$ further to obtain merely linear oversampling.

\begin{theorem}\label{thm:bsssubsample_finite_weighted} Let the assumptions from Theorem~\ref{thm:randomsubsample_finite} hold and $\bm X$ be such that \eqref{eq:mz_subs} holds.
    Further, let $b > \kappa^2$ with
    \begin{align*}
        \kappa
        = \frac{3B}{2A}+\frac 12 + \sqrt{\Big(\frac{3B}{2A}+\frac 12\Big)^2-1}\,.
    \end{align*}
    Then \texttt{BSS}-subsampling (cf.\ \cite[Algorithm~1]{BSU22}) the points $\bm X$, we obtain points $\bm X' = \{\bm x^i\}_{i\in J'} \subset \bm X$ with $|\bm X'| \le \lceil b|I|\rceil$ and non-negative weights $s_i$, $i\in J'$ such that there holds the subsampled $L_2$-MZ inequality
    \begin{align*}
        \frac{1}{2}A\|f\|_{L_2}^2
        \le \sum_{i \in J'} \frac{\omega_i s_i}{n \varrho_i} |f(\bm x^i)|^2
        \le \frac{3}{2}\frac{(\sqrt b+1)^2}{(\sqrt b-1)(\sqrt b-\kappa)}B\|f\|_{L_2}^2
        \quad\forall
        f\in \spn\{\eta_k\}_{k\in I}\,.
    \end{align*}
\end{theorem} 

\begin{proof} The result is an immediate consequence of applying \cite[Theorem~3.1]{BSU22} to the randomly subsampled points of Theorem~\ref{thm:randomsubsample_finite}.
\end{proof} 

The previous result is based on \cite{BaSpSr09} where tight frames were subsampled, which was used \cite{LiTe20} for subsampling random points.
This was then extended in \cite{BSU22} for non-tight frames as well, which we use here.
The \texttt{BSS}-algorithm gives no control over the weights $s_i$, however.
One alternative would be to use Weaver-subsampling to lose the weights, but this is highly nonconstructive and increases the involved constants tremendously, cf.\ \cite{NaSchUl20}.
A clever extension of the \texttt{BSS}-algorithm makes it possible to lose the weights, regain the constructiveness, choose a smaller oversampling factor $b$, and save the left-hand side of the $L_2$-MZ inequality, which is the important one as it allows for the reconstruction from the function evaluations $f(\bm x^i)$.

\begin{theorem}\label{thm:bsssubsample_finite} Let the assumptions from Theorem~\ref{thm:randomsubsample_finite} hold and $\bm X$ be such that \eqref{eq:mz_subs} holds.
    Further, let $I\subset \Imz$ and $b > 1+\frac{1}{|I|}$.
    Then \texttt{PlainBSS}-subsampling (cf.\ \cite[Algorithm~3]{BSU22}) the points $\bm X$, we obtain points $\bm X' = (\bm x^i)_{i\in J'} \subset \bm X$ with $|\bm X'| \le \lceil b|I|\rceil$ such that there holds the subsampled left $L_2$-MZ inequality
    \begin{align*}
        \frac{(b-1)^3}{178\,(b+1)^2}A \|f\|_{L_2}^2
        \le \frac{1}{|I|}\sum_{i \in J'} \frac{\omega_i}{\varrho_i} |f(\bm x^i)|^2
        \quad\text{for all}\quad
        f\in V\,.
    \end{align*}
\end{theorem} 

\begin{proof} The result is an immediate consequence of applying \cite[Corollary~4.5]{BSU22} to the randomly subsampled points of Theorem~\ref{thm:randomsubsample_finite}.
\end{proof}  
\section{Sampling recovery in reproducing kernel Hilbert spaces}\label{sec:general}
In this section we tackle the same problem as in Section~\ref{sec:lsqr}, namely function recovery from samples, but now for infinite-dimensional function spaces.
We consider functions from separable reproducing kernel Hilbert spaces $H(K) \subset \mathds C^D$ with Hermitian positive definite kernels $K(\cdot, \cdot)$ on $D\times D$, where $D\subset \mathds R^d$ is some domain.
Fundamental is the reproducing property
\begin{align*}
    f(\bm x)
    = \langle f, K(\cdot, \bm x) \rangle_{H(K)}
\end{align*}
for all $\bm x\in D$, which implies that point evaluations are continuous functionals on $H(K)$.
Introductory information about these spaces can be found in \cite[Chapter~4]{BeTh04} or \cite[Chapter~1]{StChr08}.
In the framework of this paper, we assume the kernels to be bounded, i.e.,
\begin{equation}\label{eq:boundedkernel}
    \|K\|_{\ell_\infty(D)}
    \coloneqq \sqrt{\sup_{\bm x\in D} K(\bm x, \bm x)}
    < \infty \,.
\end{equation}
With $\|\cdot\|_{H(K)}^2 = \langle \cdot, \cdot \rangle_{H_K}$, this condition implies $\|f\|_{\ell_\infty(D)} \le \|K\|_{\ell_\infty(D)} \cdot \|f\|_{H(K)}$, which, in other words, is the continuous embedding of $H(K)$ into the space of essentially bounded functions $\ell_\infty(D)$.
Furthermore, we assume the finite trace property
\begin{align}\label{eq:finitetrace}
    \|K\|_2^2
    \coloneqq \int_D K(\bm x,\bm x) \,\mathrm d\nu(\bm x)
    < \infty \,,
\end{align}
which, for finite measures, follows from \eqref{eq:boundedkernel}.
This integrability condition ensures the compactness of the embedding operator
$ \Id \colon H(K) \hookrightarrow L_2 $
which is Hilbert-Schmidt.
Now we consider the compact and self-adjoint linear operator $W = \Id^\ast\circ\Id \colon H(K) \to H(K)$.
Applying the spectral theorem, we obtain the representation
\begin{align*}
    W f
    = \sum_{k=1}^{\infty} \lambda_k \langle f, e_k \rangle_{H(K)} e_k\,,
\end{align*}
where $(\lambda_k)_{k=1}^{\infty}$ is the non-increasing rearrangement of the eigenvalues of $W$ and $(e_k)_{k=1}^{\infty} \subset H(K)$ the system of eigenfunctions of $W$ which form an orthonormal system in $H(K)$.
Using the reproducing property, we obtain a representation of the kernel in terms of the eigenfunctions
\begin{align}\label{eq:anvil}
    K(\bm x, \bm y)
    = \sum_{k\in\mathds N}\langle K(\bm x, \cdot), e_k\rangle_{H(K)} e_k(\bm y)
    = \sum_{k\in\mathds N}\overline{e_k(\bm x)} e_k(\bm y)\,.
\end{align}
Further, we have
\begin{align*}
    \langle e_k, e_l \rangle_{L_2}
    = \langle \Id(e_k), \Id(e_l) \rangle_{L_2}
    = \langle W(e_k), e_l \rangle_{H(K)}
    = \lambda_k \langle e_k, e_l \rangle_{H(K)}
    = \lambda_k \delta_{k,l}\,.
\end{align*}
Therefore, with $\eta_k(\bm x) \coloneqq \sigma_k\inv e_k(\bm x)$ where $\sigma_k = \sqrt{\lambda_k}$ are the singular numbers of $\Id$, we obtain an orthonomal system in $L_2$.

\begin{remark}\label{remark:spectral}
    For $V = \spn\{\eta_k\}_{k\in I}$, $\eta_k$ $L_2$-orthonormal, we have for the projection $P_{I}f$ defined in \eqref{eq:projection}
    \begin{align*}
        \sup_{\|f\|_{H(K)} \le 1} \|f-P_{I} f\|_{\ell_\infty(D)}^2
        &= \sup_{\bm x\in D} \sup_{\|f\|_{H(K)} \le 1} \Big| \sum_{k\notin I} \langle f, e_k \rangle_{H(K)} e_k(\bm x) \Big|^2 \\
        &= \sup_{\bm x\in D} \sup_{\|\bm a\|_2^2 \le 1} \Big| \sum_{k\notin I} a_k e_k(\bm x) \Big|^2 \,.
    \end{align*}
    Interpreting this as a functional on $\ell_2(\mathds N \setminus I)$ applied to $\bm a = (a_k)_{k\notin I}$, we obtain that the above equals the norm of the Riesz representer $(e_k(\bm x))_{k\notin I}$.
    Thus,
    \begin{align*}
        \sup_{\|f\|_{H(K)} \le 1} \|f-P_{I} f\|_{\ell_\infty(D)}^2
        = \sup_{\bm x\in D} \sum_{k\notin I} |e_k(\bm x)|^2 \,.
\end{align*}
\end{remark}

 \subsection{Sampling recovery with \texorpdfstring{$L_2$}{L2}-MZ points}
A reconstruction model using finitely many samples cannot reconstruct every function in an infinite-dimensional function space.
Thus, we make an error.
In particular, for least squares we reconstruct functions $f\in V = \spn\{\eta_k\}_{k\in \Imz}$ but none in $H(K)\setminus V$.
In this section we quantify the worst-case error of the least-squares method using samples in $L_2$-MZ points.
Similar results can be found in \cite{Gr20}.

We consider the following set of assumptions on the initial points $\Xmz$ and weights $\Wmz$ for the theorems of this section.

\begin{assumption}\label{ass} Let $H(K)$ be a separable reproducing kernel Hilbert space with a finite trace \eqref{eq:finitetrace} and bounded \eqref{eq:boundedkernel} kernel.
    Further, let $\lambda_k$ and $\eta_k$ denote the eigenvalues and $L_2$-orthonormal eigenfunctions of the operator $W = \Id^\ast\circ\Id$ with $\Id\colon H(K)\hookrightarrow L_2$ the embedding.
    
    For $V = \spn\{\eta_k\}_{k\in\Imz} \subset H(K)$ a finite-dimensional function space, we assume the points $\Xmz\subset D$ and weights $\Wmz\in [0,\infty)^{M\times M}$ fulfill an $L_2$-MZ inequality for $V$ with constants $A$ and $B$.
\end{assumption} 

\begin{theorem}\label{thm:quaderror} Let Assumptions~\ref{ass} hold.
    Then the error of the reconstruction operator $S_{\Imz}^{\Xmz} = S_{\Imz}^{\Xmz}(\Wmz)$ is zero for functions in $V=\spn\{\eta_k\}_{k\in\Imz}$ and otherwise bounded as follows:
    \begin{align*}
        \sup_{\|f\|_{H(K)} \le 1} \Big\|f-S_{\Imz}^{\Xmz}f\Big\|_{L_2}^2
        \le& \sup_{k\notin \Imz} \lambda_k + \frac{\sum_{i=1}^{M} \omega_i}{A} \sup_{\|f\|_{H(K)}\le 1} \|f-P_{\Imz}f\|_{\ell_\infty(D)}^2 \,.
    \end{align*}
\end{theorem} 

\begin{proof} The reconstructing property for functions in $V$ is immediate.
    For functions $f\in H(K)$ we use orthogonality of the projection \eqref{eq:projection} to obtain
    \begin{align}\label{eq:honululu}
        \|f-S_{\Imz}^{\Xmz}f\|_{L_2}^2
        = \|f-P_{\Imz} f\|_{L_2}^2 + \|P_{\Imz}f-S_{\Imz}^{\Xmz}f\|_{L_2}^2
    \end{align}
    of which we now estimate each summand individually.

    \textbf{Step 1.}
    We bound the first summand of \eqref{eq:honululu} by
    \begin{align*}
        \|f-P_{\Imz} f\|_{L_2}^2
        &= \Big\| \sum_{k\notin \Imz} \langle f, e_k \rangle_{H(K)} e_k \Big\|_{L_2}^2
        = \sum_{k\notin \Imz} \lambda_k |\langle f, e_k \rangle_{H(K)}|^2 \\
        &\le \|f\|_{H(K)}^2\sup_{k\notin \Imz} \lambda_k\,.
    \end{align*}
    
    \textbf{Step 2.}
    For the second summand of \eqref{eq:honululu} we use the reconstructing property  of $S_{\Imz}^{\Xmz}$ for functions in $V$
    \begin{align*}
        &\|P_{\Imz}f-S_{\Imz}^{\Xmz}f\|_{L_2}^2
        = \|S_{\Imz}^{\Xmz} (P_{\Imz}f-f) \|_{L_2}^2 \\
        &\,\le \Big\| ((\bm L_{\Xmz}^{\Imz})^\ast\Wmz\bm L_{\Xmz}^{\Imz})\inv(\bm L_{\Xmz}^{\Imz})^\ast\Wmz^{1/2}\Big\|_{2\to 2}^2 \Big\|\Wmz^{1/2} \Big((P_{\Imz}f-f)(\bm x^i)\Big)_{i=1}^M \Big\|_2^2\,.
    \end{align*}
    By Lemmata~\ref{prop2} and \ref{lemma:character} we have for the first factor
    \begin{align*}
        \Big\| ((\bm L_{\Xmz}^{\Imz})^\ast\Wmz\bm L_{\Xmz}^{\Imz})\inv(\bm L_{\Xmz}^{\Imz})^\ast\Wmz^{1/2}\Big\|_{2\to 2}^2
        \le \frac{1}{A}\,.
    \end{align*}
    Finally, we obtain
    \begin{align*}
        \sup_{\|f\|_{H(K)} \le 1}
        \|P_{\Imz}f-S_{\Imz}^{\Xmz}f\|_{L_2}^2
        &\le
        \frac{1}{A}
        \sup_{\|f\|_{H(K)} \le 1}
        \sum_{i=1}^M \omega_i |(P_{\Imz}f-f)(\bm x^i)|^2 \\
        &\le
        \frac{\sum_{i=1}^M \omega_i}{A}
        \sup_{\|f\|_{H(K)} \le 1}
        \|P_{\Imz}f-f\|_{\ell_\infty(D)}^2 \,.\qedhere
    \end{align*}
\end{proof} 

The above are rather basic bounds which hold in a general setting.
In special cases like the exponential functions, Chebyshev polynomials, and the half-period cosine functions this topic is examined in detail in \cite{KuoMiNoNu19} for rank-1 lattices.
In particular for function spaces with mixed smoothness there exist almost tight error bounds in \cite{ByKaUlVo16} which improve on Theorem~\ref{thm:quaderror}.
 \label{sec:mzrecovery}
\subsection{Sampling recovery with subsampled \texorpdfstring{$\mathrm L_2$}{L2}-MZ points}
The next step is to subsample the $L_2$-MZ points like in Section~\ref{sec:submz} whilst paying attention to the approximation error.
With that we obtain better error bounds and will relate the error using the subsampled points to the error of the initial point set.
In contrast to Section~\ref{sec:lsqr}, we subsample two frames and one Bessel sequence simultaneously from the points $\Xmz$.
In order to control the subsampling we use a change of measure via a convex combination of the respective probability densities.
Each point $\bm x^i$ from $\Xmz$ will be drawn with the following probability density:
\begin{align}\label{eq:density}
    \varrho_i
    = \frac{\displaystyle\omega_i\sum_{k\in I}|\eta_k(\bm x^i)|^2}{\displaystyle3\sum_{j=1}^M\omega_j\sum_{k\in I}|\eta_k(\bm x^j)|^2}
    + \frac{\displaystyle\omega_i\sum_{k\in \Imz\setminus I} |e_k(\bm x^i)|^2 }{\displaystyle3\sum_{j=1}^M\omega_j\sum_{k\in \Imz\setminus I} |e_k(\bm x^j)|^2} 
    + \frac{\omega_i}{3} \,.
\end{align}
Since $\varrho_i \ge 0$ and $\sum_{i=1}^M \varrho_i = 1$, these are proper density weights.
This is based on an idea first used in \cite{KrUl19}.

For the later analysis we need some preparations, starting with a concentration inequality for random infinite matrices from \cite{MoUl20}.

\begin{proposition}[{\cite[Proposition~3.8]{MoUl20}}]\label{MoUl20} Let $\bm u^1, \dots, \bm u^n$ be i.i.d.\ random sequences from $\ell_2(\mathds N)$.
    Let further $n\ge 3$, $t>0$, $M>0$ such that $\|\bm u^i\|_2 \le M$ almost surely and $\mathds E(\bm u^i\otimes\bm u^i) = \bm\Lambda$ for all $i=1,\dots, n$.
    Then
    \begin{align*}
        \Big\|\frac 1n \sum_{i=1}^{n} \bm u^i\otimes\bm u^i \Big\|_{2\to 2}
        \le \frac{21(\log(n)+t)}{n} M^2 + 2 \|\bm\Lambda\|_{2\to 2}
    \end{align*}
    with probability exceeding $1-2^{3/4}\exp(-t)$.
\end{proposition} 

\begin{lemma}\label{lemma:infinite} Let Assumptions~\ref{ass} hold and let $\bm X = \{\bm x^i\}_{i\in J}$, $|J| = n$, be points drawn i.i.d.\,from $\Xmz$ with respect to the discrete density weights $\varrho_i$, defined in \eqref{eq:density}.
    Further, for $I\subset\Imz$, we define the matrices
    \begin{align*}
        \bm W = \diag(\omega_i/\varrho_i)_{i\in J}
        \quad\text{and}\quad
        \bm{\Phi}^{I'}_{\bm X}
        = (e_k(\bm x^i))_{i \in J,  k \in \Imz\setminus I} \,.
    \end{align*}
    Then we have with probability exceeding $1-2^{3/4} \exp(-t)$
    \begin{align*}
        \frac 1n \Big\| \bm W^{1/2}\bm\Phi_{\bm X}^{\Imz\setminus I} \Big\|_{2\to 2}^2
        &\le \frac{63(\log(n) + t)}{n} B\sum_{k\in \Imz\setminus I} \lambda_k + 2 B \sup_{k\notin I} \lambda_k \,.
    \end{align*}
\end{lemma} 

\begin{proof} Since points with probability $\varrho_i = 0$ will not be drawn, we assume $\varrho_i > 0$.
    Let
    \begin{align*}
        \bm L_{\Xmz}^{\Imz\setminus I}
        = \Big(\eta_k(\bm x^i)\Big)_{\bm x^i \in \Xmz,  k \in \Imz\setminus I}
        = \bm\Phi_{\Xmz}^{\Imz\setminus I}\diag(\sigma_k\inv)_{k \in \Imz\setminus I}.
    \end{align*}
    We want to apply Proposition~\ref{MoUl20}.
    To this end we define $\bm u^i = \sqrt{\omega_i / \varrho_i} (e_k(\bm x^i))_{k\in \Imz\setminus I}$.
    Note,
    \begin{align*}
        \sum_{i\in J} \bm u^i\otimes\bm u^i
        = \Big( \bm\Phi_{\bm X}^{\Imz\setminus I} \Big)\herm \bm W \bm\Phi_{\bm X}^{\Imz\setminus I} \,.
    \end{align*}
    Now we estimate $\|\bm u^i\|_2^2$ and $\|\mathds E \bm u^i\otimes\bm u^i\|_{2\to 2}$.
    
    We bound $\|\bm u^i\|_2$ by using the $L_2$-MZ inequality and $\|e_k\|_{L_2}^2 = \lambda_k$
    \begin{align*}
        \|\bm u^i\|_2^2
        \le 3\sum_{k\in \Imz\setminus I} \sum_{j=1}^{M} \omega_j |e_k(\bm x^j)|^2
        \le 3B\sum_{k\in \Imz\setminus I} \|e_k\|_{L_2}^2
= 3B \sum_{ k\in \Imz\setminus I} \lambda_k\,.
    \end{align*}
    We further have by the compatibility of the spectral norm and Lemma~\ref{lemma:character}
    \begin{align*}
        \Big\| \mathds E \bm u^i\otimes\bm u^i \Big\|_{2\to 2}
        &= \Big\| \Wmz^{1/2}\bm\Phi_{\Xmz}^{\Imz\setminus I} \Big\|_{2\to 2}^2 \\
        &\le \Big\| \Wmz^{1/2}\bm L_{\Xmz}^{\Imz\setminus I} \Big\|_{2\to 2}^2 \Big\|\diag(\sigma_k)_{\Imz\setminus I}\Big\|_{2\to 2}^2 \\
        &\le B \sup_{k\in \Imz\setminus I}\lambda_k\,,
    \end{align*}
    which gives the assertion after applying Proposition~\ref{MoUl20}.
\end{proof} 

We now formulate a central result of the paper, which bounds the worst-case reconstruction error for the least squares method where the points are drawn randomly from a discrete set of points fulfilling an $L_2$-MZ inequality, cf.\ middle of Figure~\ref{fig:scheme}.

\begin{theorem}\label{thm:general} Let Assumption~\ref{ass} hold and $I \subset  \Imz \subset \mathbb N$, $|I| \ge 3$, be an index set.
    For $n\in\mathds N$, $t>0$, and $r\ge 1$ such that
    \begin{align*}
        n
        \coloneqq \Big\lceil \frac{36 B}{A} |I| (\log |I|+t) \Big\rceil
        \le |I|^r,
    \end{align*}
    let $\bm X = (\bm x^i)_{i\in J}$, $|J| = n$, be points drawn i.i.d.\,from $\Xmz$ with respect to the discrete density weights $\varrho_i$, defined in \eqref{eq:density}.
    Then we have with probability exceeding $1 - 4\exp(-t)$
    \begin{align*}
        \sup_{\|f\|_{H(K)}\le 1} \!\!\! \|f-S_{I}^{\bm X}f\|_{L_2}^2
        &\le
        \frac{9B}{A} \sup_{ k\notin I}\lambda_k
        +
        \frac{7r}{|I|}\! \sum_{k\in \Imz\setminus I} \lambda_k \\
        &\quad+\frac{12}{A} \sup_{\|f\|_{H(K)} \le 1} \|f-P_{\Imz} f\|_{\ell_\infty(D)}^2 \,,
    \end{align*}
    where $\bm W = \diag(\omega_i/\varrho_i)_{i\in J}$ and $S_{I}^{\bm X} = S_{I}^{\bm X}(\bm W)$.
\end{theorem} 

\begin{proof} Without loss of generality let $\varrho_i > 0$.
    We begin by defining two events.
    The first one bounds the spectral norm of the least squares matrix
    \begin{align*}
        S_1 = \Big\{ \bm X\subset \Xmz : \Big\|\Big(\bm L\herm \bm W \bm L\Big)\inv \bm L\herm \bm W^{1/2} \Big\|_{2\to 2}^2 \le \frac{2}{A n} \Big\}.
    \end{align*}
    Using Lemma~\ref{prop2}, we have, that this is equivalent to a lower bound of the singular values of $\bm W^{1/2}\bm L$.
    Via Lemma~\ref{lemma:character} we use the $L_2$-MZ characterization and the assumption on $n$ in order to apply Theorem~\ref{thm:randomsubsample_finite} with $C=1/3$ and obtain that the above holds with probability $1-2\exp(-t)$.
  
    For $\bm Y\in\{\Xmz,\bm X\}$ we define the matrix
    \begin{align*}
        \bm \Phi^{\Imz\setminus I}_{\bm Y} 
        = \Big(e_k(\bm x^i)\Big)_{\bm x^i \in \bm Y,  k \in \Imz\setminus I} \,.
    \end{align*}
    Event $S_2$ is the inequalities of Lemma~\ref{lemma:infinite} which holds with probability at least $1-2^{3/4}\exp(-t)$.
    Together we obtain the desired probability
    \begin{align*}
        \mathds P (S_1\cap S_2)
        \ge 1 - \mathds P S_1^\complement - \mathds P S_2^\complement
        \ge 1 - 4\exp(-t)\,.
    \end{align*}
    It remains to show the bound based on the events $S_1$ and $S_2$.
    On that account we decompose the recovery error using triangle inequality
    \begin{align}\label{eq:sum}
        \|f-S_{I}^{\bm X}f\|_{L_2}^2
        &\le \|f-P_{I}f\|_{L_2}^2 + \|P_{ I}f-S_{I}^{\bm X}f\|_{L_2}^2 \nonumber\\
        &\le \|f-P_{I}f\|_{L_2}^2 + 2 \|P_{ I}f-S_{I}^{\bm X}P_{ \Imz}f\|_{L_2}^2 + 2 \|S_{I}^{\bm X}P_{\Imz}f-S_{I}^{\bm X}f\|_{L_2}^2\,.
    \end{align}
    In the following we estimate each of the three summands individually.
  
    \textbf{Step 1.}
    We estimate the first summand of \eqref{eq:sum} in the same manner as in the first step of the proof of Theorem~\ref{thm:quaderror}:
    \begin{align*}
        \|f-P_I f\|_{L_2}^2
        &\le \|f\|_{H(K)}^2\sup_{ k\notin \mathcal I} \lambda_k\,.
    \end{align*}
  
    \textbf{Step 2.} Using the invariance of $S_I^{\bm X}$ on $P_I f$, we estimate the second summand of \eqref{eq:sum} by
    \begin{align*}
        \Big\|P_If-S_I^{\bm X} P_{\Imz}f\Big\|_{L_2}^2
        &= \Big\|S_{I}^{\bm X}(P_{\Imz}f-P_{I}f)\Big\|_{L_2}^2 \\
        &= \Big\| (\bm L^\ast \bm W \bm L)\inv \bm L^\ast\bm W ((P_{\Imz}f-P_I f)(\bm x^i))_{i\in J}\Big\|_2^2 \\
        &\le \Big\| (\bm L^\ast \bm W \bm L)\inv \bm L^\ast \bm W^{1/2}\Big\|_{2\to 2} \sum_{i\in J} \frac{\omega_i}{\varrho_i} |(P_{\Imz}f-P_I f)(\bm x^i)|^2\,.
    \end{align*}
    We estimate the first factor by the means of event $S_1$ and write the latter in terms of its coefficients in order to apply event $S_2$
    \begin{align*}
        &\Big\|S_{I}^{\bm X}P_{ \Imz}f-S_{I}^{\bm X}f\Big\|_{L_2}^2
        \le \frac{2}{A n} \Big\|\bm W^{1/2} \bm{\Phi}_{\Imz\setminus I}^{\bm X} \Big(\langle f, e_k\rangle_{H(K)}\Big)_{ k\in \Imz\setminus I}\Big\|_2^2 \\
        &\quad\le \frac{2}{A n} \Big\|\bm W^{1/2} \bm\Phi_{\Imz\setminus I}^{\bm X}\Big\|_{2\to 2}^2 \sum_{ k\in \Imz\setminus I} |\langle f, e_k\rangle_{H(K)}|^2 \\
        &\quad\le
        \Big( \frac{126 B (\log(n)+t)}{A n} \sum_{k\in \Imz\setminus I} \lambda_k + \frac{4B}{A} \sup_{k\notin I} \lambda_k \Big)
        \|f\|_{H(K)}^2\,.
    \end{align*}
  
    \textbf{Step 3.} We start estimating the third summand of \eqref{eq:sum} analogously to Step~2
    \begin{align*}
        \|P_{I}f- S_{I}^{\bm X}P_{ \Imz}f\|_{L_2}^2
        &\le \frac{2}{An} \sum_{i\in J}\frac{\omega_i}{\varrho_i} |(f-P_{\Imz}f)(\bm x^i)|^2 \,.
    \end{align*}
    Now we use the third part of $\varrho_i$, to obtain
    \begin{align*}
        \|P_{I}f- S_{I}^{\bm X}P_{ \Imz}f\|_{L_2}^2
        &\le \frac{6}{An} \sum_{i\in J} |(f-P_{\Imz}f)(\bm x^i)|^{2} \\
        &\le \frac{6}{A} \|f-P_{\Imz}f\|_{\ell_\infty(D)}^{2} \,.
    \end{align*}
  
    \textbf{Overall}, applying the estimates 1.\,-- 3.\,in \eqref{eq:sum} we obtain
    \begin{align*}
        &\sup_{\|f\|_{H(K)}\le 1} \|f-S_{I}^{\bm X}f\|_{L_2}^2
        \le
        \sup_{ k\notin I}\lambda_k
        +
        \frac{252B(\log(n)+t)}{A n} \sum_{k\in \Imz\setminus I} \lambda_k + \frac{8B}{A} \sup_{k\notin I} \lambda_k \\
        &\quad + \frac{12}{A} \sup_{\|f\|_{H(K)}\le 1} \|f-P_{\Imz}f\|_{\ell_\infty(D)}^2 \,.
    \end{align*}
    By the assumption on $|I|$ we have
    \begin{align*}
        \frac{252 B (\log(n)+t)}{A n}
        \le \frac{7 (\log(n)+t)}{|I|(\log(|I|)+t)}
        \le \frac{7r (\log(|I|)+t)}{|I|(\log(|I|)+t)}
        =\frac{7r}{|I|}
        \le 3 \,,
    \end{align*}
    and obtain the assertion.
\end{proof} 

\begin{remark}\label{rem:aftercentral} \begin{enumerate}[(i)]
    \item
        Given a bounded orthonormal system, i.e., $\|\eta_k\|_{\ell_\infty(D)} = C < \infty$ for all $k$ we can alter events $S_2$ and $S_3$ such that we won't need to sample with respect to the density weights $\varrho_i$ and sample with respect to the weights $\omega_i$ coming from the $L_2$-MZ inequality directly.
    \item
        Whenever the Christoffel-function $\sum_{k\in I} |\eta_k(\bm x)|^2$ is constant, independent of the underlying index set $I$, the three summands in density weights $\varrho_i$ in \eqref{eq:density} coincide and the number of random points may be divided by three whilst achieving the same error bound.
    \item
        Up to the quotient of the constants for the $L_2$-MZ inequality, the first two summands of Theorem~\ref{thm:general} are the same as in \cite{KUV19} where points were drawn with respect to a continuous probability measure.
        The difference is in the latter summand, which only depends on the initial point set.
        By choosing a suitable initial point set, i.e., a point set satisfying a $L_2$-MZ inequality for large enough $\Imz$, we can make this as small as needed.
        In particular smaller than the first two terms, which, therefore, determine the error decay behaviour.
\end{enumerate}
\end{remark} 

Next, we use the unweighted frame subsampling from \cite{BSU22} to prove the main result which lowers the number of points to be linear in $|I|$, cf.\ right of Figure~\ref{fig:scheme}.

\begin{theorem}\label{thm:generalbss} Let Assumption~\ref{ass} hold and $I \subset  \Imz \subset \mathbb N$, $|I| \ge 3$ be an index set.
    For $n\in\mathds N$, $t>0$, and $r\ge 1$ such that
    \begin{align*}
        n
        \coloneqq \left\lceil \frac{36 B}{A} |I| (\log |I|+t) \right\rceil
        \le |I|^r,
    \end{align*}
    let $\bm X = \{\bm x^i\}_{i\in J}$, $|J| = n$ be points drawn i.i.d.\,from $\Xmz$ with respect to the discrete density weights $\varrho_i$, defined in \eqref{eq:density}.
    For $b > 1+\frac{1}{|I|}$, \texttt{PlainBSS} subsampling (cf.\ \cite[Algorithm~3]{BSU22}) the points $\bm X$, we obtain points $\bm X' = \{\bm x^i\}_{i\in J'} \subset \bm X$ with $|\bm X'| \le \lceil b|I|\rceil$ such that we have with probability exceeding $1 - 4\exp(-t)$
    \begin{align*}
        &\sup_{\|f\|_{H(K)}\le 1} \|f-S_{I}^{\bm X'}f\|_{L_2}^2 \\
        &\quad\le
        39\,808\,\frac{B}{A} \frac{(b+1)^2}{(b-1)^3} (\log|I|+t)
        \Bigg(
        \frac{B}{A} \sup_{ k\notin I}\lambda_k
        +
        \frac{r}{|I|}\sum_{k\in\Imz\setminus I} \lambda_k
        \\
        &\quad + \frac{1}{A} \sup_{\|f\|_{H(K)}\le 1} \|f-P_{\Imz}f\|_{\ell_\infty(D)}^2
        \Bigg)
    \end{align*}
    with weights $\bm W' = \diag(\omega_i/\varrho_i)_{i\in J'}$ and the least squares approximation $S_{I}^{\bm X'}f = S_{I}^{\bm X'}(\bm W')f$ defined in \eqref{eq:lsqr}.
\end{theorem} 

\begin{proof} For $g\in V$ we have
    \begin{align*}
        \|S_I^{\bm X'} g\|_{L_2}^2
        \le \Big\|((\bm L')\herm\bm W'\bm L')\inv(\bm L')\herm(\bm W')^{1/2} \Big\|_{2\to 2}^2 \sum_{i\in J'} \frac{\omega_i}{\varrho_i} |g(\bm x^i)|^2 \,.
    \end{align*}
    By Theorem~\ref{thm:bsssubsample_finite} and Lemma~\ref{prop2} we have
    \begin{align*}
        \|S_I^{\bm X'} g\|_{L_2}^2
        &\le \frac{89(b+1)^2}{(b-1)^3} \frac{n}{|I|} \frac{2}{A n} \sum_{i\in J'} \frac{\omega_i}{\varrho_i} |g(\bm x^i)|^2 \\
        &\le \frac{3234(b+1)^2 B}{(b-1)^3 A} (\log(|I|) + t) \Big( \frac{2}{A n} \sum_{i\in J'} \frac{\omega_i}{\varrho_i} |g(\bm x^i)|^2 \Big) \,,
    \end{align*}
    where $|I|\ge 3$ was used in
    \begin{align*}
        n = \Big\lceil \frac{36B}{A}|I|(\log|I|+t)\Big\rceil
        \le \frac{109B}{3A}|I|(\log|I|+t) \,.
    \end{align*}
    Using this estimate in step 2 and 3 of the proof of Theorem~\ref{thm:general} we obtain the assertion.
\end{proof} 

We payed the logarithmic factor in the bound to work with linearly many points whilst achieving the same error bound as in \cite{NaSchUl20}.
Recent progress has shown that the logarithmic factor can be avoided but this utilizes the Kadison-Singer theorem and is not constructive, cf.\ \cite{DKU22}.
 \label{sec:submzrecovery}

\section{Application: Sampling recovery on \texorpdfstring{$\mathds T^d$}{Td} with rank-1 lattices}\label{sec:t}
In this section we apply our general theory from Section~\ref{sec:submzrecovery} to the $d$-dimensional torus $\mathds T^d\simeq (\mathds R/\mathds Z)^d$.
First we deal with function spaces on $\mathds T^d$.
To begin with, we consider a weight function $w(\mathbf{k}) \colon \mathds Z^d\to(0,\infty)$ for $\mathbf{k} \in \mathds{Z}^d$ and define the kernel 
\begin{align}\label{t:kerneld}
  K_w(\mathbf{x},\mathbf{y}):=\sum_{\mathbf{k}\in \mathds Z^d} \frac{\exp(2\pi\mathrm{i}\langle\mathbf{k},(\mathbf{y}-\mathbf{x})\rangle )}{w(\bm k)^2}
  \quad\text{for}\quad
  \mathbf{x},\mathbf{y}\in {\mathds T}^d\,,
\end{align}
which is a reproducing kernel for $H^w := H(K_w)$.
The kernel defined in \eqref{t:kerneld} is associated to the inner product 
\begin{align}\label{eq:seed}
  \langle f , g\rangle_{H^w} := \sum_{\bm k\in \mathds Z^d} \hat{f}_{\bm k} \, \overline{\hat{g}_{\bm k}} \,w(\bm k)^2
\end{align}
and the corresponding norm.
These definitions make sense if $\sum_{\bm k \in \mathds{Z^d}} w(\bm k)^{-2}<\infty$. 

An interesting case are Sobolev spaces with mixed smoothness.
For $s\in \mathds N$ we define the space $H^{s}_{\textnormal{mix}} = H^{s}_{\textnormal{mix}}(\mathds T^d)$ as the Hilbert space with the inner product
\begin{align}\label{eq:hsmix}
  \langle f , g\rangle_{H^{s}_{\textnormal{mix}}} :=   \sum_{\bm j \in \{0,s\}^d}\langle D^{(\bm j)}f,D^{(\bm j)}g \rangle_{L_2}\,,
  \quad\text{with}\quad
  D^{(\bm j)} = \frac{\partial^{\|j\|_1}}{\partial x_1^{j_1} \cdots \partial x_d^{j_d}} \,.
\end{align}
In the univariate case the kernel equals
\begin{align*}
  K^1_{s}(x,y):=\sum_{k\in \mathds Z} \frac{\exp(2\pi\mathrm{i}k(y-x) )}{w_{s}(k)^2}
  \quad\text{with}\quad
  w_{s}(k) = \Big(1+(2\pi |k|)^{2s}\Big)^{1/2}
\end{align*}
for $k\in \mathds Z$ and $x,y\in \mathds T$.
This directly leads to the $d$-dimensional kernel
\begin{align*}
  K^d_{s}(\bm x,\bm y):= K^1_{s}(x_1,y_1)\otimes \cdots \otimes   K^1_{s}(x_d,y_d)
  \quad\text{for}\quad \bm x,\bm y \in \mathds T^d\,.
\end{align*}
This construction of the kernel extends to $s\in\mathds R$, $s>1/2$ naturally.
These function spaces are suitable for high-dimensional approximation.
By collecting the most important frequencies with respect to the singular values we end up with a special structure for the frequency index set, namely a hyperbolic cross.
Detailed information on approximation in these spaces can be found in \cite{Tem93, DTU18}.

Next, we deal with $L_2$-MZ inequalities on the torus.
In particular we are interested in those coming from exact quadrature rules, cf.~Lemma~\ref{lemma:mzquadrature}.
The canonical example would be equispaced points
\begin{align*}
    \Xmz
    = \Big\{ \frac{1}{\sqrt[d]{M}} \bm i : \bm i \in \{1, \dots, \sqrt[d]{M} \}^d \Big\}
\end{align*}
(assuming $\sqrt[d]{M}$ is integer) which fulfill an $L_2$-MZ inequality with equal weights $\omega_i = 1/M$ and $A=B$ for the frequency index set
$\Imz = \{-\sqrt[d]{M}/2, \dots, \sqrt[d]{M}/2-1\}^d$.
However, these points are unfeasible in high-dimensional approximation if we do not subsample.
Another candidate for quadrature rules suitable for Sobolev spaces with mixed smoothness are rank-1 lattices
\begin{equation*}
  \Xmz
  \coloneqq \Big\{ \frac 1M (i\bm z \!\!\!\!\mod M\mathds 1) \in\mathds T^d : i = 0,\dots,M-1\Big\}
\end{equation*}
which were studied in \cite{Sloan94, kaemmererdiss, KaPoVo13, PlPoStTa18, DKP22}.
We say a rank-1 lattices $\Xmz = \{\bm x^1, \dots, \bm x^M\}$ has a reconstructing property for a frequency index set $\Imz$ when
\begin{equation}\label{eq:reconstructingproperty}
    \frac{1}{M} \sum_{i=1}^{M} \exp(2\pi\mathrm i\langle \bm k-\bm l, \bm x^i \rangle)
    = \delta_{\bm l, \bm k} \,,
\end{equation}
for all $\bm k, \bm l\in\Imz$, which is equivalent to an $L_2$-MZ inequality \eqref{eq:mz} for $A=B$ with equal weights $\omega_i$.
There exist algorithms which, given a frequency index set $\Imz$ and $M$, compute a generating vector $\bm z$ such that the rank-1 lattice $\Xmz$ has this reconstructing property, cf. \cite{kaemmererdiss,KuoMiNoNu19}.
In particular, the probabilistic approach recently presented in \cite{Kaemmerer20} seems to be particularly suitable for downstream subsampling. On the one hand the component-by-component (CBC) rank-1 lattice construction is extremely efficient with computational costs that are linear in $|\Imz|$ up to a few logarithmic factors. On the other hand, the sizes $M$ of the resulting rank-1 lattices might be slightly (but only up to a factor of two with very high probability) larger than those resulting from the deterministic approaches. In summary, the CBC construction costs are not a major concern, and random subsampling the rank-1 lattice reduces the number of samples actually used to a number that does not depend on the slightly too large $M$.

Next, we consider asymptotic approximation results for rank-1 lattices in Sobolev spaces with dominating mixed smoothness from \cite{ByKaUlVo16}.
For simplicity we omit constants in the following.
For two sequences $(a_n)_{n=1}^{\infty}$ and $(b_n)_{n=1}^\infty \subset \mathds R$ we write $a_n \lesssim b_n$ if there exists a constant $c>0$ such that $a_n \le c b_n$ for all $n$.
We will write $a_n \asymp b_n$ if $a_n\lesssim b_n$ and $b_n \lesssim a_n$.

\begin{theorem}[\cite{KSU15, BDSU16, ByKaUlVo16} Approximation with rank-1 lattices in $H_{\textnormal{mix}}^s$]\label{thm:r1} Let $\Imz \subset \mathbb N$ be a hyperbolic cross frequency index set, i.e.,
    \begin{align}\label{eq:hc}
        \Imz = \Big\{\bm k\in\mathds Z^d : \prod_{j=1}^{d} \max\Big\{1, \frac{|k_j|}{\gamma}\Big\} \le R \Big\}
    \end{align}
    for some $\gamma > 0$, $R\in(1,\infty)$, and $\Xmz$ a corresponding reconstructing rank-1 lattice with $M$ points, cf.\ \eqref{eq:reconstructingproperty}.
    \begin{enumerate}[(i)]
    \item
        The size of a hyperbolic cross is asymptotically $|\Imz| \asymp R (\log R)^{d-1}$.
    \item
        The best possible error of functions in $H_{\textnormal{mix}}^s$ behaves as follows
        \begin{align*}
            \sup_{\|f\|_{H_{\textnormal{mix}}^s}\le 1} \|f-P_{\Imz}f\|_{L_2}^2
            \asymp |\Imz|^{-2s} (\log|\Imz|)^{2(d-1)s}
        \end{align*}
        and
        \begin{align*}
            \sup_{\|f\|_{H_{\textnormal{mix}}^s}\le 1} \|f-P_{\Imz}f\|_{\ell_{\infty}(\mathds T^d)}^2
            \asymp |\Imz|^{-2s+1/2} (\log|\Imz|)^{2(d-1)s} \,.
        \end{align*}
    \item
        The size of a reconstructing rank-1 lattice for $\Imz$ is bounded by
        \begin{align*}
            |\Imz|^2(\log|\Imz|)^{-2(d-1)}
            \lesssim M
            \lesssim |\Imz|^2 (\log |\Imz|)^{-d} \,,
        \end{align*}
        where the lower inequality holds for all reconstructing rank-1 lattices and there exists a rank-1 lattice fulfilling the upper one.
    \item
        The error of the least squares operator using rank-1 lattices is bounded as follows:
        \begin{align*}
            M^{-s}
            \lesssim \sup_{\|f\|_{H_{\textnormal{mix}}^s}\le 1} \|f-S_{\Imz}^{\Xmz} f\|_{L_2}^2
            \lesssim \begin{cases}
            |\Imz|^{-2s}(\log|\Imz|)^{(d-1)(2s+1)} \\
            M^{-s}(\log M)^{(d-2)s+d-1} \,,
            \end{cases}
        \end{align*}
        where the lower inequality holds for all reconstructing rank-1 lattices and there exists a rank-1 lattice fulfilling the upper one.
    \end{enumerate}
\end{theorem} 

\begin{proof} There are different definitions for the hyperbolic cross and the Sobolev spaces with mixed smoothness.
    Up to constants the considered quantities coincide, cf.\ \cite{KSU15}.
    The assertion (i) is given in \cite[Lemma~2]{ByKaUlVo16}.
    To show (ii), we use $\sup_{\|f\|_{H_{\textnormal{mix}}^s}\le 1} \|f-P_{\Imz}f\|_{L_2}^2 = R^{-s}$ and (i).
    The second part of (ii) is given in \cite[Theorem~6.11~(iii)]{BDSU16}.
    
    Assertion (iii) follows from Lemmata~2 and 3 in \cite{ByKaUlVo16}.
    To show (iv), we use \cite[Theorem~2]{ByKaUlVo16} with $\alpha = s$, $\beta = \gamma = 0$ and Lemmata~2 and 3 from \cite{ByKaUlVo16} again.
    
    Note, that the Fourier coefficients of the approximation in \cite{ByKaUlVo16} (formula~2.3) are computed by applying the adjoint of the Fourier matrix.
    Because of the reconstructing property, this coincides with the least squares approximation presented here, i.e., with $\bm W = \diag(1/M, \dots, 1/M)$ we have
    \begin{equation*}
        S_{\Imz}^{\Xmz} f
        = \sum_{\bm k\in\Imz}a_{\bm k}\exp(2\pi\mathrm i \langle\bm k,\bm x\rangle)
        \quad\text{with}\quad
        \bm a
        = (\bm L^\ast \bm W\bm L)\inv \bm L^\ast\bm W\bm f
        = \frac{1}{M} \bm L^\ast\bm f \,.
        \qedhere
    \end{equation*}
\end{proof} 

Similar bounds were obtained for the kernel method operating on rank-1 lattice points, cf.\ \cite{KKKS21}.
For random points one has the rate $M^{-2s}$ and additional logarithmic terms, however.
So, using full rank-1 lattices we lose half the rate of convergence in the main order.
But we recover the order of convergence by subsampling as stated in Corollary~\ref{cor:r1}.

\begin{proof}[Proof of Corollary~\ref{cor:r1}] Since we are dealing with a reconstructing lattice, we have a tight frame $A=B$.
We need to apply Theorem~\ref{thm:generalbss} and Remark~\ref{rem:aftercentral} and use the following two inequalities
\begin{align*}
   \sum_{\bm k\notin I} \lambda_{\bm k}
   = \sum_{\bm k\notin I} \prod_{j=1}^{d} w_s(k_j)^{-2}
   \lesssim |I|^{-2s+1}(\log|I|)^{2(d-1)s}
\end{align*}
from \cite[(2.3.2)]{DTU18}.
Using Hölder's inequality and the above inequality again, we obtain
\begin{align*}
    \|f-P_{\Imz}f\|_{\ell_\infty}^2
    &\le \Big| \sum_{k\notin \Imz} \lambda_{\bm k}\inv \lambda_{\bm k} |\hat f_{\bm k}| \Big|^2 \\
    &\le \sum_{k\notin \Imz} \lambda_{\bm k}^{-2} \sum_{k\notin\Imz} \lambda_{\bm k}^2 |\hat f_{\bm k}|^2 \\
    &\lesssim |\Imz|^{-2s+1}(\log|\Imz|)^{2(d-1)s} \|f\|_{H^s_{\textnormal{mix}}}^2 \,.\qedhere
\end{align*}
\end{proof} 

\begin{remark}[Optimality for rank-1 lattices] With the initial hyperbolic cross $\Imz$ in Corollary~\ref{cor:r1} slightly bigger than $I$ we recover the optimal error bound $n^{-s}(\log n)^{(d-1)s}$, cf.\ \cite[(2.3.2)]{DTU18}, up to a logarithmic factor.
    Thus, the phrase ``This main rate (without logarithmic factors) is half the optimal main rate [...] and turns out to be the best possible among all algorithms taking samples on lattices'' from \cite{ByKaUlVo16} has to be stated more precisely to ``\dots samples on \emph{full} lattices''.
    We conjecture that the non-constructive approach in \cite{DKU22} (based on Kadison-Singer and Weaver subsampling \cite{NaSchUl20}) may lead to a bound without additional logarithm.
\end{remark}

\subsection{Numerical experiments}\label{subsec:t_numerics} 

Because of the one-dimensional structure of \mbox{rank-1} lattices, the matrix-vector product with the system matrix $\bm L_{\Xmz}$ can be carried out using a one-dimensional fast Fourier transform in $\mathcal O(M\log M)$.
We may use this algorithm for the subsampled points $\bm X = \{\bm x^i\}_{i\in J}$ and system matrix $\bm L$ too:
\vspace{-0.4cm}
\begin{align}\label{eq:mask}
  \bm L
  = \bm P \bm L_{\Xmz}
  \quad\text{where}\quad
  \bm P 
  = \left(\begin{matrix}\\\\\\\\\end{matrix}\right. \begin{matrix} \\ 1 & 0 & 0 & \dots & 0 & 0 & 0 \\ 0 & 0 & 1 & \dots & 0 & 0 & 0 \\ & \vdots &&&& \vdots & \\ 0 & 0 & 0 & \dots & 0 & 1 & 0 \\ \textcolor{gray}{j_1} && \textcolor{gray}{j_2} &&& \textcolor{gray}{j_n} \end{matrix} \left.\begin{matrix}\\\\\\\\\end{matrix}\right)
    \textcolor{gray}{\begin{matrix}1\\2\\\vdots\\n\end{matrix}}\,.
\end{align}
\vspace{-0.1cm}

Here, we use hyperbolic crosses for the frequency index sets with size $M\lesssim |\Imz|^2(\log|\Imz|)^{-d}$, cf.\ Theorem~\ref{thm:r1}~(iii).
Using the Fast Fourier Transform this yields the following complexity for the matrix-vector product with the full rank-1 lattice and, subsequently, the subsampled one:
\begin{equation*}
    \mathcal O(M\log M)
    = \mathcal O( |\Imz|^2(\log|\Imz|)^{1-d})
\end{equation*}
with the same memory usage.
In contrast, if we naively set up the matrix with the $n \sim |\Imz|\log|\Imz|$ random points we have a complexity of
$ \mathcal O(|\Imz|^2\log|\Imz|) $
for the number of arithmetic operations and memory usage of the matrix-vector product.
If we further use \texttt{BSS} subsampling this reduces to $\mathcal O(|\Imz|^2)$, which is still slower than using the Fast Fourier Transform.
In general, whenever we start with a rank-1 lattice with fewer than quadratic points, we gain computation speed.

\paragraph{Experiment 1.} For the numerical experiments, we follow \cite{ByKaUlVo16, KUV19} and use the five-dimensional kink function
\begin{align*}
  f(\bm x)
  = f((x_1, \dots, x_d)\transp)
  = \prod_{j=1}^5 \Big( \frac{5^{3/4} 15}{4\sqrt 3} \max\Big\{\frac 15 - \Big(x_j-\frac 12\Big)^2, 0\Big\}\Big)
\end{align*}
of which we know the exact Fourier coefficients.
We remark that $f\in H^{3/2-\varepsilon}_{\text{mix}}(\mathds T^5)$ for $\varepsilon > 0$ and $\|f\|_{L_2} = 1$.

\begin{figure}
  \centering
  \begingroup
  \makeatletter
  \providecommand\color[2][]{\GenericError{(gnuplot) \space\space\space\@spaces}{Package color not loaded in conjunction with
      terminal option `colourtext'}{See the gnuplot documentation for explanation.}{Either use 'blacktext' in gnuplot or load the package
      color.sty in LaTeX.}\renewcommand\color[2][]{}}\providecommand\includegraphics[2][]{\GenericError{(gnuplot) \space\space\space\@spaces}{Package graphicx or graphics not loaded}{See the gnuplot documentation for explanation.}{The gnuplot epslatex terminal needs graphicx.sty or graphics.sty.}\renewcommand\includegraphics[2][]{}}\providecommand\rotatebox[2]{#2}\@ifundefined{ifGPcolor}{\newif\ifGPcolor
    \GPcolortrue
  }{}\@ifundefined{ifGPblacktext}{\newif\ifGPblacktext
    \GPblacktexttrue
  }{}\let\gplgaddtomacro\g@addto@macro
\gdef\gplbacktext{}\gdef\gplfronttext{}\makeatother
  \ifGPblacktext
\def\colorrgb#1{}\def\colorgray#1{}\else
\ifGPcolor
      \def\colorrgb#1{\color[rgb]{#1}}\def\colorgray#1{\color[gray]{#1}}\expandafter\def\csname LTw\endcsname{\color{white}}\expandafter\def\csname LTb\endcsname{\color{black}}\expandafter\def\csname LTa\endcsname{\color{black}}\expandafter\def\csname LT0\endcsname{\color[rgb]{1,0,0}}\expandafter\def\csname LT1\endcsname{\color[rgb]{0,1,0}}\expandafter\def\csname LT2\endcsname{\color[rgb]{0,0,1}}\expandafter\def\csname LT3\endcsname{\color[rgb]{1,0,1}}\expandafter\def\csname LT4\endcsname{\color[rgb]{0,1,1}}\expandafter\def\csname LT5\endcsname{\color[rgb]{1,1,0}}\expandafter\def\csname LT6\endcsname{\color[rgb]{0,0,0}}\expandafter\def\csname LT7\endcsname{\color[rgb]{1,0.3,0}}\expandafter\def\csname LT8\endcsname{\color[rgb]{0.5,0.5,0.5}}\else
\def\colorrgb#1{\color{black}}\def\colorgray#1{\color[gray]{#1}}\expandafter\def\csname LTw\endcsname{\color{white}}\expandafter\def\csname LTb\endcsname{\color{black}}\expandafter\def\csname LTa\endcsname{\color{black}}\expandafter\def\csname LT0\endcsname{\color{black}}\expandafter\def\csname LT1\endcsname{\color{black}}\expandafter\def\csname LT2\endcsname{\color{black}}\expandafter\def\csname LT3\endcsname{\color{black}}\expandafter\def\csname LT4\endcsname{\color{black}}\expandafter\def\csname LT5\endcsname{\color{black}}\expandafter\def\csname LT6\endcsname{\color{black}}\expandafter\def\csname LT7\endcsname{\color{black}}\expandafter\def\csname LT8\endcsname{\color{black}}\fi
  \fi
    \setlength{\unitlength}{0.0500bp}\ifx\gptboxheight\undefined \newlength{\gptboxheight}\newlength{\gptboxwidth}\newsavebox{\gptboxtext}\fi \setlength{\fboxrule}{0.5pt}\setlength{\fboxsep}{1pt}\definecolor{tbcol}{rgb}{1,1,1}\begin{picture}(7360.00,5660.00)\gplgaddtomacro\gplbacktext{\csname LTb\endcsname \put(750,3434){\makebox(0,0)[r]{\strut{}\scriptsize $10^{-13}$}}\csname LTb\endcsname \put(750,3775){\makebox(0,0)[r]{\strut{}\scriptsize $10^{-10}$}}\csname LTb\endcsname \put(750,4116){\makebox(0,0)[r]{\strut{}\scriptsize $10^{-7}$}}\csname LTb\endcsname \put(750,4457){\makebox(0,0)[r]{\strut{}\scriptsize $10^{-4}$}}\csname LTb\endcsname \put(750,4798){\makebox(0,0)[r]{\strut{}\scriptsize $10^{-1}$}}\csname LTb\endcsname \put(1312,3291){\makebox(0,0){\strut{}\scriptsize $10^{2}$}}\csname LTb\endcsname \put(2368,3291){\makebox(0,0){\strut{}\scriptsize $10^{4}$}}\csname LTb\endcsname \put(3424,3291){\makebox(0,0){\strut{}\scriptsize $10^{6}$}}}\gplgaddtomacro\gplfronttext{\csname LTb\endcsname \put(2114,3127){\makebox(0,0){\strut{}\scriptsize \# of frequencies}}\csname LTb\endcsname \put(2114,5332){\makebox(0,0){\strut{}truncation and aliasing $L_2$-error}}}\gplgaddtomacro\gplbacktext{\csname LTb\endcsname \put(4420,3434){\makebox(0,0)[r]{\strut{}\scriptsize $10^{0}$}}\csname LTb\endcsname \put(4420,3788){\makebox(0,0)[r]{\strut{}\scriptsize $10^{2}$}}\csname LTb\endcsname \put(4420,4141){\makebox(0,0)[r]{\strut{}\scriptsize $10^{4}$}}\csname LTb\endcsname \put(4420,4495){\makebox(0,0)[r]{\strut{}\scriptsize $10^{6}$}}\csname LTb\endcsname \put(4420,4848){\makebox(0,0)[r]{\strut{}\scriptsize $10^{8}$}}\csname LTb\endcsname \put(4982,3291){\makebox(0,0){\strut{}\scriptsize $10^{2}$}}\csname LTb\endcsname \put(6038,3291){\makebox(0,0){\strut{}\scriptsize $10^{4}$}}\csname LTb\endcsname \put(7094,3291){\makebox(0,0){\strut{}\scriptsize $10^{6}$}}}\gplgaddtomacro\gplfronttext{\csname LTb\endcsname \put(5784,3127){\makebox(0,0){\strut{}\scriptsize \# of frequencies}}\csname LTb\endcsname \put(5784,5332){\makebox(0,0){\strut{}\# of points}}}\gplgaddtomacro\gplbacktext{\csname LTb\endcsname \put(750,614){\makebox(0,0)[r]{\strut{}\scriptsize $10^{-8}$}}\csname LTb\endcsname \put(750,968){\makebox(0,0)[r]{\strut{}\scriptsize $10^{-6}$}}\csname LTb\endcsname \put(750,1321){\makebox(0,0)[r]{\strut{}\scriptsize $10^{-4}$}}\csname LTb\endcsname \put(750,1675){\makebox(0,0)[r]{\strut{}\scriptsize $10^{-2}$}}\csname LTb\endcsname \put(750,2028){\makebox(0,0)[r]{\strut{}\scriptsize $10^{0}$}}\csname LTb\endcsname \put(1115,471){\makebox(0,0){\strut{}\scriptsize $10^{2}$}}\csname LTb\endcsname \put(1777,471){\makebox(0,0){\strut{}\scriptsize $10^{4}$}}\csname LTb\endcsname \put(2439,471){\makebox(0,0){\strut{}\scriptsize $10^{6}$}}\csname LTb\endcsname \put(3101,471){\makebox(0,0){\strut{}\scriptsize $10^{8}$}}}\gplgaddtomacro\gplfronttext{\csname LTb\endcsname \put(2114,307){\makebox(0,0){\strut{}\scriptsize \# of points}}\csname LTb\endcsname \put(2114,2512){\makebox(0,0){\strut{}$L_2$-error}}}\gplgaddtomacro\gplbacktext{\csname LTb\endcsname \put(4420,614){\makebox(0,0)[r]{\strut{}\scriptsize $10^{-4}$}}\csname LTb\endcsname \put(4420,968){\makebox(0,0)[r]{\strut{}\scriptsize $10^{-2}$}}\csname LTb\endcsname \put(4420,1321){\makebox(0,0)[r]{\strut{}\scriptsize $10^{0}$}}\csname LTb\endcsname \put(4420,1675){\makebox(0,0)[r]{\strut{}\scriptsize $10^{2}$}}\csname LTb\endcsname \put(4420,2028){\makebox(0,0)[r]{\strut{}\scriptsize $10^{4}$}}\csname LTb\endcsname \put(4982,471){\makebox(0,0){\strut{}\scriptsize $10^{2}$}}\csname LTb\endcsname \put(6038,471){\makebox(0,0){\strut{}\scriptsize $10^{4}$}}\csname LTb\endcsname \put(7094,471){\makebox(0,0){\strut{}\scriptsize $10^{6}$}}}\gplgaddtomacro\gplfronttext{\csname LTb\endcsname \put(5784,307){\makebox(0,0){\strut{}\scriptsize \# of frequencies}}\csname LTb\endcsname \put(5784,2512){\makebox(0,0){\strut{}time in seconds}}}\gplbacktext
    \put(0,0){\includegraphics[width={368.00bp},height={283.00bp}]{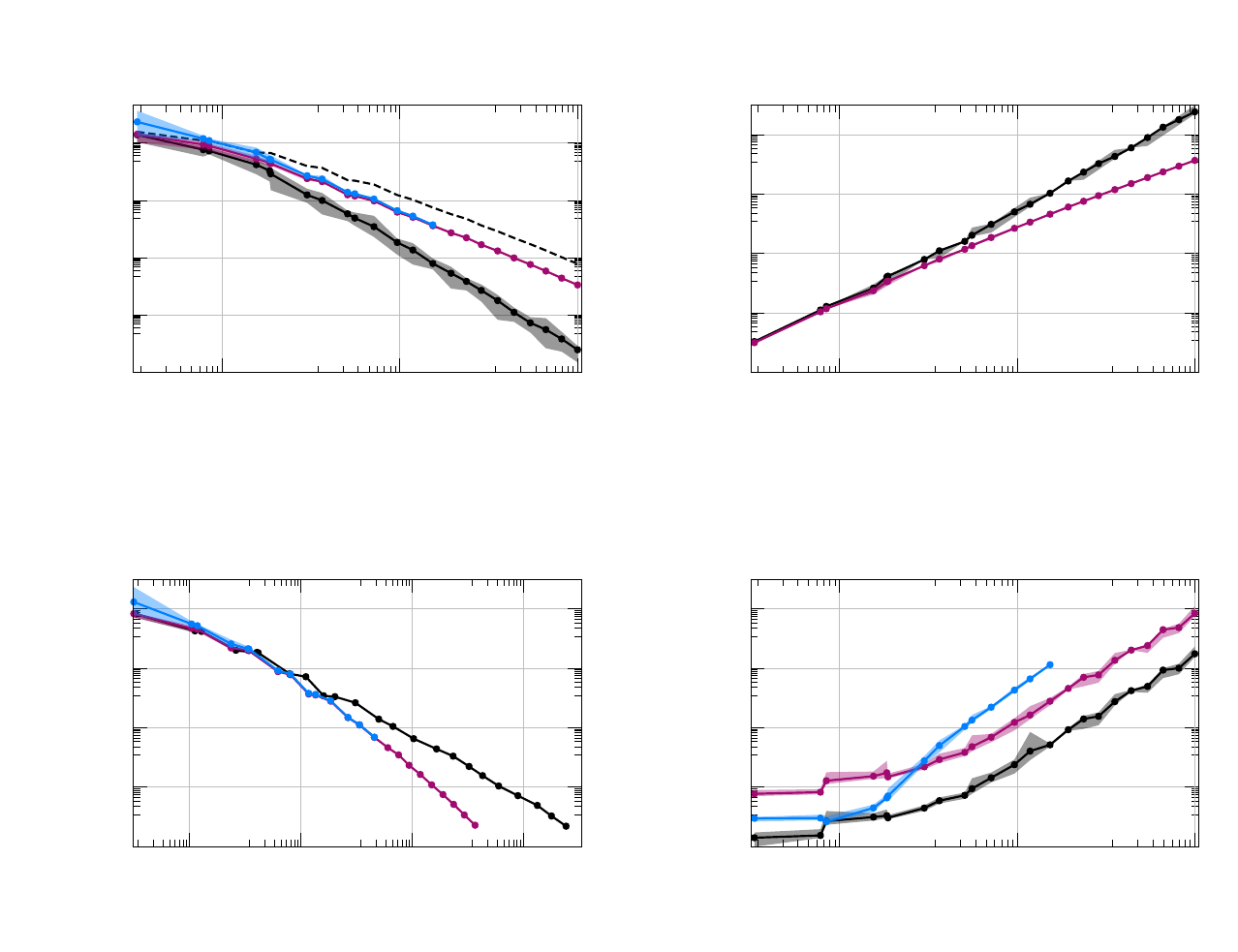}}\gplfronttext
  \end{picture}\endgroup
   \caption{Five-dimensional experiment~1 on the torus for different point sets and algorithms. Black: full lattice, \textcolor{signalviolet}{magenta:} randomly subsampled lattice, \textcolor{azure}{azure:} continuously random points, and dashed the truncation error.}
  \label{fig:experiment1_t}
\end{figure}

For frequency sets we use hyperbolic crosses with different radii $R$, defined in \eqref{eq:hc}, where we choose the shape parameter $\gamma = 1/2$. Using the probabilistic algorithm from \cite{Kaemmerer20} we computed a reconstructing rank-1 lattice $\Xmz$ for $I$.
We used three different techniques for approximation:
\begin{itemize}
\item 
    We then solved the least squares system \eqref{eq:lsqr} for the full rank-1 lattice.
    Using Lemma~\ref{lemma:mzquadrature} for $\bm L^\ast\bm W\bm L = \bm I$, we omit the inverse matrix.
    The special structure of the points was used to perform the matrix-vector product using a one-dimensional fast Fourier transform in $\mathcal O(M\log M)$.
\item
    We randomly subsampled the full rank-1 lattices according to Theorem~\ref{thm:general} (all points have equal probability $\varrho_i = 1/M$) such that we have $n = \lceil |I|\log|I|\rceil$ points , cf.~Figure~\ref{fig:scheme}.
    We solved the least squares system \eqref{eq:lsqr} iteratively using the same one-dimensional fast Fourier transform as described in \eqref{eq:mask} and set an iteration limit of $10$.
\item
    For comparison, we drew random samples with respect to the Lebesgue measure as suggested in \cite{KUV19} with the same number of points $n = \lceil |I|\log |I| \rceil$.
    We set up the system matrix and solved the least squares system \eqref{eq:lsqr} with the $n$ random points iteratively.
    We limit the iterations to at most $10$ as before.
\end{itemize}
We computed the $L_2$-error decomposed into truncation error
\begin{align*}
    \|f-P_I f\|_{L_2}^2 = \|f\|_{L_2}^2-\sum_{k\in I}|\hat f_k|^2
\end{align*}
and the aliasing error
\begin{align*}
    \|P_I f - S_{I}^{X} f\|_{L_2}^{2} = \sum_{k\in I}|\hat f_k-\hat g_k|^2
\end{align*}
with $\hat g_k$ the Fourier coefficients of the approximation.
Further we measured the elapsed time for the computations and stopped the computations when more then $100\,\text{gigabytes}$ of memory were used.
We repeated this experiment ten times and the minimal, maximal and average  results can be seen in Figure~\ref{fig:experiment1_t}.

\begin{itemize}
\item
    In the upper left figure, we see that for all proposed methods the aliasing error is below the truncation error (dashed line). We emphasize at this point that we have chosen $I=\Imz$, which is of course reasonable since the bottleneck is the truncation error, which cannot be prevented.
    Differences in the sets $I$ and $\Imz$ might only have a positive effect on the aliasing error, which seems to be superfluous in our setting.
    
    Note, not to loose the main rate in Theorem~\ref{thm:r1} for $s=3/2$, we would need to choose $|\Imz|\gtrsim |I|^{3/2}$ or $M\gtrsim n^3$.
    That we still observe the optimal main rate in practice leads us to suspect that the theory might be improved by suitably replacing the $\ell_\infty$-term.
\item
    In the upper right figure, we see that the full rank-$1$ lattice has more points and is worth subsampling.
    In particular for $10^6$ frequencies we have an oversampling factor of $489$ for the full lattice whereas the oversampling for the random choices is $\log(10^6)\le 14$.
\item
    In the lower left we see the slower error decay of the full rank-1 lattice and the faster error decay of the continuously random points with respect to the number of points.
    The error of the subsampled rank-1 lattice is the minimum of these two.
    In particular, for increasing number of points we obtain the better error decay for random points as well as for the randomly subsampled rank-1 lattice.
\item
    In the lower right we see the computation time of the respective methods.
    The full rank-1 lattice and the subsampled rank-1 lattice differ by a factor of ten.
    This is due to the same underlying algorithm and the iteration count of ten for the subsampled version.
    Longer than both took the continuous random points as there the full matrix has to be used in contrast to using the highly tuned Fast Fourier Transform.
    As we showed at the beginning of this section this is slower in regards to the complexity ($\mathcal O(|I|^2(\log|I|)^{1-d})$ versus $\mathcal O(|I|\log|I|)$).
    For $23\,483$ frequencies the computations took $258$ seconds for the full matrix, whereas merely $4$ seconds for the subsampled rank-1 lattice.
    Furthermore storing the matrix needs $23\,483\cdot 212\,432\cdot 16\,\text{bytes} \approx 80\,\text{gigabytes}$, where $16\,\text{bytes}$ is the size of a complex floating point number with double precision (which is why the experiment for the continuously random points stopped early).
    In contrast to that the rank-1 lattice uses around $1\,000\,000$ points and therefore $1\,000\,000\cdot 16\,\text{bytes} = 16\,\text{megabytes}$ of memory.
    Even the largest considered rank-1 lattice uses merely $8\,\text{gigabytes}$.
\end{itemize}

\paragraph{Experiment~2.} We repeat Experiment~1 but additionally \texttt{BSS} subsample the randomly subsampled rank-1 lattice further to an oversampling factor of $b=2$
(the \texttt{BSS} algorithm is available at \texttt{https://github.com/felixbartel/BSSsubsampling.jl}).
We need to stop earlier as the \texttt{BSS} algorithm is (so far) not suitable for arbitrarily large matrices.
The results can be seen in Figure~\ref{fig:experiment2_t}.
\begin{figure}
  \centering
  \begingroup
  \makeatletter
  \providecommand\color[2][]{\GenericError{(gnuplot) \space\space\space\@spaces}{Package color not loaded in conjunction with
      terminal option `colourtext'}{See the gnuplot documentation for explanation.}{Either use 'blacktext' in gnuplot or load the package
      color.sty in LaTeX.}\renewcommand\color[2][]{}}\providecommand\includegraphics[2][]{\GenericError{(gnuplot) \space\space\space\@spaces}{Package graphicx or graphics not loaded}{See the gnuplot documentation for explanation.}{The gnuplot epslatex terminal needs graphicx.sty or graphics.sty.}\renewcommand\includegraphics[2][]{}}\providecommand\rotatebox[2]{#2}\@ifundefined{ifGPcolor}{\newif\ifGPcolor
    \GPcolortrue
  }{}\@ifundefined{ifGPblacktext}{\newif\ifGPblacktext
    \GPblacktexttrue
  }{}\let\gplgaddtomacro\g@addto@macro
\gdef\gplbacktext{}\gdef\gplfronttext{}\makeatother
  \ifGPblacktext
\def\colorrgb#1{}\def\colorgray#1{}\else
\ifGPcolor
      \def\colorrgb#1{\color[rgb]{#1}}\def\colorgray#1{\color[gray]{#1}}\expandafter\def\csname LTw\endcsname{\color{white}}\expandafter\def\csname LTb\endcsname{\color{black}}\expandafter\def\csname LTa\endcsname{\color{black}}\expandafter\def\csname LT0\endcsname{\color[rgb]{1,0,0}}\expandafter\def\csname LT1\endcsname{\color[rgb]{0,1,0}}\expandafter\def\csname LT2\endcsname{\color[rgb]{0,0,1}}\expandafter\def\csname LT3\endcsname{\color[rgb]{1,0,1}}\expandafter\def\csname LT4\endcsname{\color[rgb]{0,1,1}}\expandafter\def\csname LT5\endcsname{\color[rgb]{1,1,0}}\expandafter\def\csname LT6\endcsname{\color[rgb]{0,0,0}}\expandafter\def\csname LT7\endcsname{\color[rgb]{1,0.3,0}}\expandafter\def\csname LT8\endcsname{\color[rgb]{0.5,0.5,0.5}}\else
\def\colorrgb#1{\color{black}}\def\colorgray#1{\color[gray]{#1}}\expandafter\def\csname LTw\endcsname{\color{white}}\expandafter\def\csname LTb\endcsname{\color{black}}\expandafter\def\csname LTa\endcsname{\color{black}}\expandafter\def\csname LT0\endcsname{\color{black}}\expandafter\def\csname LT1\endcsname{\color{black}}\expandafter\def\csname LT2\endcsname{\color{black}}\expandafter\def\csname LT3\endcsname{\color{black}}\expandafter\def\csname LT4\endcsname{\color{black}}\expandafter\def\csname LT5\endcsname{\color{black}}\expandafter\def\csname LT6\endcsname{\color{black}}\expandafter\def\csname LT7\endcsname{\color{black}}\expandafter\def\csname LT8\endcsname{\color{black}}\fi
  \fi
    \setlength{\unitlength}{0.0500bp}\ifx\gptboxheight\undefined \newlength{\gptboxheight}\newlength{\gptboxwidth}\newsavebox{\gptboxtext}\fi \setlength{\fboxrule}{0.5pt}\setlength{\fboxsep}{1pt}\definecolor{tbcol}{rgb}{1,1,1}\begin{picture}(7360.00,5660.00)\gplgaddtomacro\gplbacktext{\csname LTb\endcsname \put(750,3434){\makebox(0,0)[r]{\strut{}\scriptsize $10^{-6}$}}\csname LTb\endcsname \put(750,3699){\makebox(0,0)[r]{\strut{}\scriptsize $10^{-5}$}}\csname LTb\endcsname \put(750,3964){\makebox(0,0)[r]{\strut{}\scriptsize $10^{-4}$}}\csname LTb\endcsname \put(750,4230){\makebox(0,0)[r]{\strut{}\scriptsize $10^{-3}$}}\csname LTb\endcsname \put(750,4495){\makebox(0,0)[r]{\strut{}\scriptsize $10^{-2}$}}\csname LTb\endcsname \put(750,4760){\makebox(0,0)[r]{\strut{}\scriptsize $10^{-1}$}}\csname LTb\endcsname \put(750,5025){\makebox(0,0)[r]{\strut{}\scriptsize $10^{0}$}}\csname LTb\endcsname \put(784,3291){\makebox(0,0){\strut{}\scriptsize $10^{1}$}}\csname LTb\endcsname \put(1671,3291){\makebox(0,0){\strut{}\scriptsize $10^{2}$}}\csname LTb\endcsname \put(2558,3291){\makebox(0,0){\strut{}\scriptsize $10^{3}$}}\csname LTb\endcsname \put(3445,3291){\makebox(0,0){\strut{}\scriptsize $10^{4}$}}}\gplgaddtomacro\gplfronttext{\csname LTb\endcsname \put(2114,3127){\makebox(0,0){\strut{}\scriptsize \# of frequencies}}\csname LTb\endcsname \put(2114,5332){\makebox(0,0){\strut{}truncation and aliasing $L_2$-error}}}\gplgaddtomacro\gplbacktext{\csname LTb\endcsname \put(4420,3434){\makebox(0,0)[r]{\strut{}\scriptsize $10^{0}$}}\csname LTb\endcsname \put(4420,3752){\makebox(0,0)[r]{\strut{}\scriptsize $10^{1}$}}\csname LTb\endcsname \put(4420,4070){\makebox(0,0)[r]{\strut{}\scriptsize $10^{2}$}}\csname LTb\endcsname \put(4420,4389){\makebox(0,0)[r]{\strut{}\scriptsize $10^{3}$}}\csname LTb\endcsname \put(4420,4707){\makebox(0,0)[r]{\strut{}\scriptsize $10^{4}$}}\csname LTb\endcsname \put(4420,5025){\makebox(0,0)[r]{\strut{}\scriptsize $10^{5}$}}\csname LTb\endcsname \put(4454,3291){\makebox(0,0){\strut{}\scriptsize $10^{1}$}}\csname LTb\endcsname \put(5341,3291){\makebox(0,0){\strut{}\scriptsize $10^{2}$}}\csname LTb\endcsname \put(6228,3291){\makebox(0,0){\strut{}\scriptsize $10^{3}$}}\csname LTb\endcsname \put(7115,3291){\makebox(0,0){\strut{}\scriptsize $10^{4}$}}}\gplgaddtomacro\gplfronttext{\csname LTb\endcsname \put(5784,3127){\makebox(0,0){\strut{}\scriptsize \# of frequencies}}\csname LTb\endcsname \put(5784,5332){\makebox(0,0){\strut{}\# of points}}}\gplgaddtomacro\gplbacktext{\csname LTb\endcsname \put(750,614){\makebox(0,0)[r]{\strut{}\scriptsize $10^{-3}$}}\csname LTb\endcsname \put(750,1144){\makebox(0,0)[r]{\strut{}\scriptsize $10^{-2}$}}\csname LTb\endcsname \put(750,1675){\makebox(0,0)[r]{\strut{}\scriptsize $10^{-1}$}}\csname LTb\endcsname \put(750,2205){\makebox(0,0)[r]{\strut{}\scriptsize $10^{0}$}}\csname LTb\endcsname \put(784,471){\makebox(0,0){\strut{}\scriptsize $10^{1}$}}\csname LTb\endcsname \put(1449,471){\makebox(0,0){\strut{}\scriptsize $10^{2}$}}\csname LTb\endcsname \put(2115,471){\makebox(0,0){\strut{}\scriptsize $10^{3}$}}\csname LTb\endcsname \put(2780,471){\makebox(0,0){\strut{}\scriptsize $10^{4}$}}\csname LTb\endcsname \put(3445,471){\makebox(0,0){\strut{}\scriptsize $10^{5}$}}}\gplgaddtomacro\gplfronttext{\csname LTb\endcsname \put(2114,307){\makebox(0,0){\strut{}\scriptsize \# of points}}\csname LTb\endcsname \put(2114,2512){\makebox(0,0){\strut{}$L_2$-error}}}\gplgaddtomacro\gplbacktext{\csname LTb\endcsname \put(4420,614){\makebox(0,0)[r]{\strut{}\scriptsize $10^{-4}$}}\csname LTb\endcsname \put(4420,932){\makebox(0,0)[r]{\strut{}\scriptsize $10^{-2}$}}\csname LTb\endcsname \put(4420,1250){\makebox(0,0)[r]{\strut{}\scriptsize $10^{0}$}}\csname LTb\endcsname \put(4420,1569){\makebox(0,0)[r]{\strut{}\scriptsize $10^{2}$}}\csname LTb\endcsname \put(4420,1887){\makebox(0,0)[r]{\strut{}\scriptsize $10^{4}$}}\csname LTb\endcsname \put(4420,2205){\makebox(0,0)[r]{\strut{}\scriptsize $10^{6}$}}\csname LTb\endcsname \put(4454,471){\makebox(0,0){\strut{}\scriptsize $10^{1}$}}\csname LTb\endcsname \put(5341,471){\makebox(0,0){\strut{}\scriptsize $10^{2}$}}\csname LTb\endcsname \put(6228,471){\makebox(0,0){\strut{}\scriptsize $10^{3}$}}\csname LTb\endcsname \put(7115,471){\makebox(0,0){\strut{}\scriptsize $10^{4}$}}}\gplgaddtomacro\gplfronttext{\csname LTb\endcsname \put(5784,307){\makebox(0,0){\strut{}\scriptsize \# of frequencies}}\csname LTb\endcsname \put(5784,2512){\makebox(0,0){\strut{}time in seconds}}}\gplbacktext
    \put(0,0){\includegraphics[width={368.00bp},height={283.00bp}]{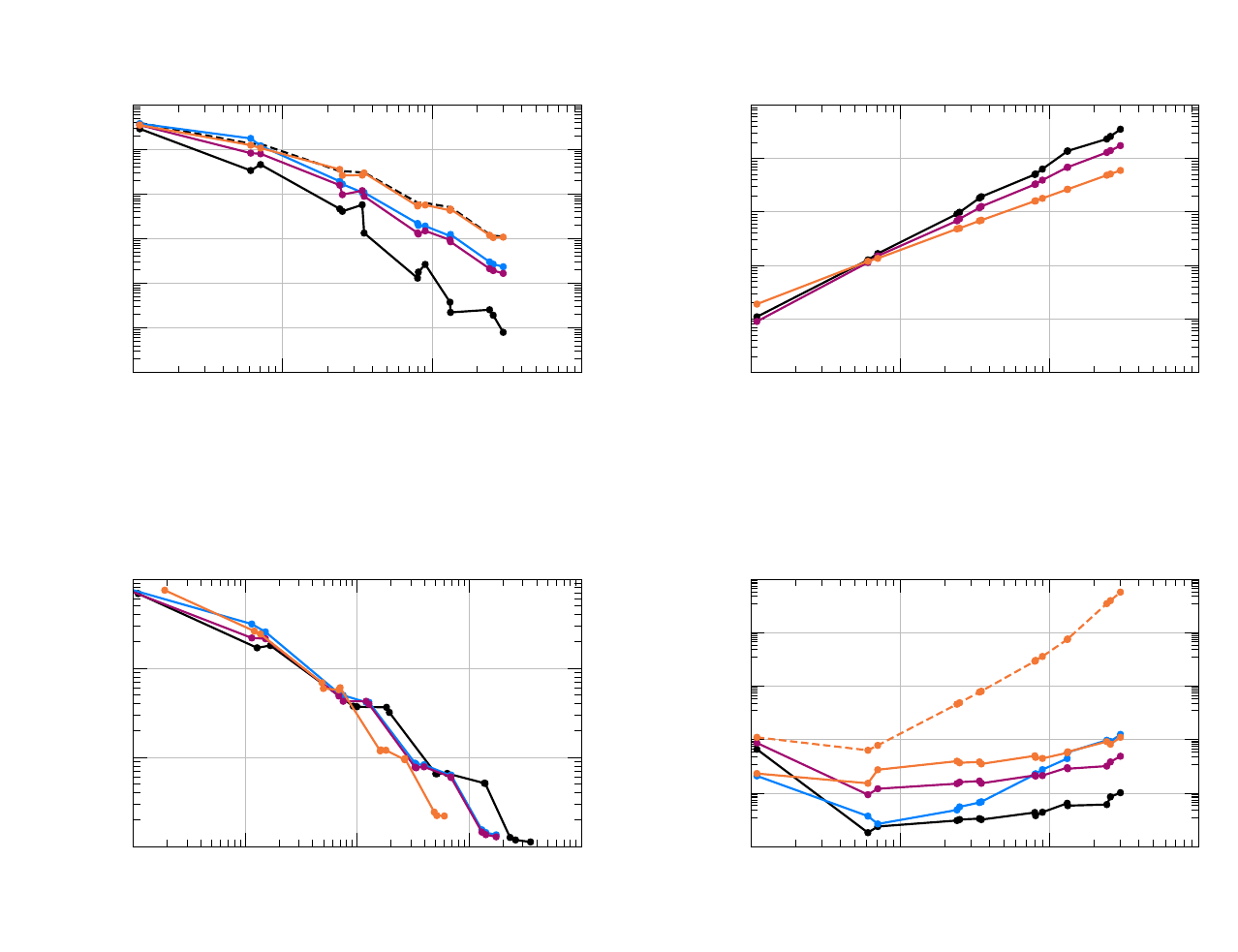}}\gplfronttext
  \end{picture}\endgroup
   \caption{Five-dimensional experiment~2 on the torus for different point sets and algorithms. Black: full lattice, \textcolor{signalviolet}{magenta:} randomly subsampled lattice, \textcolor{orange}{orange:} random and \texttt{BSS} subsampled lattice (dashed: time of the \texttt{BSS} precomputation step) \textcolor{azure}{azure:} continuously random points, and dashed the truncation error}
  \label{fig:experiment2_t}
\end{figure}

As the random subsampling step was already evaluated in the experiment above, we focus on the \texttt{BSS} subsampling.
We obtain an even smaller number of points ($|\bm X'| = 2|I|$) while still having the aliasing error slightly smaller than the truncation error.
This results in a faster error decay with respect to the number of points, cf.\ lower left of Figure~\ref{fig:experiment2_t}.
We also see that the \texttt{BSS} algorithm takes way more time, cf. lower right of Figure~\ref{fig:experiment2_t}.
But this is a precomputation step and only has to be done once with the actual iterative solver for the solution not suffering from this. 
\section*{Acknowledgement}
Felix~Bartel would like to acknowledge the support by the BMBF grant 01-S20053A (project SA$\ell$E). Daniel Potts acknowledge funding by Deutsche Forschungsgemeinschaft (DFG, German Research Foundation) -- Project-ID 416228727 -- SFB 1410. Lutz K\"ammerer acknowledges funding by the Deutsche Forschungsgemeinschaft (DFG, German Research Foundation) -- project number 38064826.

\bibliographystyle{amsplain}
\providecommand{\bysame}{\leavevmode\hbox to3em{\hrulefill}\thinspace}
\providecommand{\MR}{\relax\ifhmode\unskip\space\fi MR }
\providecommand{\MRhref}[2]{\href{http://www.ams.org/mathscinet-getitem?mr=#1}{#2}
}
\providecommand{\href}[2]{#2}

\end{document}